\documentclass[11pt,reqno]{amsart}
  
  \usepackage[utf8]{inputenc}
  \usepackage[a4paper,left=2.5cm,top=3cm,bottom=3cm,right=2.5cm]{geometry}
  \usepackage{parskip}
  \usepackage{mathptmx}

  \usepackage[usenames,dvipsnames]{xcolor} 
  \usepackage[colorlinks=true,citecolor=PineGreen,linkcolor=RedOrange,urlcolor=NavyBlue]{hyperref}

  \usepackage{amssymb,enumerate,eucal,tikz,extarrows,stmaryrd,titlesec,titletoc} 
  \usepackage[all]{xy}

  \titleformat{\section}
    {\normalfont\normalsize\bfseries\filcenter}{\thesection.}{1em}{}
  \titleformat{\subsection}
    {\normalfont\normalsize\bfseries}{\thesubsection.}{1em}{}

  \dottedcontents{section}[0.6cm]{}{0.6cm}{.3pc}
  \dottedcontents{subsection}[1.4cm]{}{0.8cm}{0pc}

  \newcommand{\cM}{\mathcal{M}}
  \newcommand{\cO}{\mathcal{O}}
  \newcommand{\CC}{\mathbf{C}}
  \newcommand{\ZZ}{\mathbf{Z}}
  \newcommand{\NN}{\mathbf{N}}

  \renewcommand{\emptyset}{\varnothing}
  \newcommand{\ov}{\overline}

  \newcommand{\op}{\operatorname}
  \renewcommand{\bigwedge}{\mbox{\large $\wedge$}}
  \DeclareMathOperator{\Hom}{Hom}
  \DeclareMathOperator{\Spec}{Spec}
  \newcommand{\cat}{\mathrm} 
  \renewcommand{\coprod}{\bigsqcup} 
 
  \newcommand{\CCt}{\CC(\!(t)\!)}
  \newcommand{\CCs}{\CC[\![t]\!]}
  \newcommand{\bLambda}{\Lambda} 
  \newcommand{\oM}{\ov{\cM}}
  \newcommand{\eps}{\varepsilon}
  \newcommand{\epsast}{\eps^\circledast}
  \newcommand{\gp}{\mathrm{gp}}
  \newcommand{\MIC}{\cat{MIC}}

  \newcommand{\xto}{\xrightarrow}
  \newcommand{\ra}{\longrightarrow}
  \newcommand{\isomto}{\xlongrightarrow{\,\smash{\raisebox{-0.65ex}{\ensuremath{\displaystyle\sim}}}\,}}

  \theoremstyle{plain}
  \newtheorem{lemma}{Lemma}[section]
  \newtheorem{cor}[lemma]{Corollary}
  \newtheorem{prop}[lemma]{Proposition}
  \newtheorem{thm}[lemma]{Theorem}
  \newtheorem{conj}[lemma]{Conjecture}

  \theoremstyle{definition}
  \newtheorem{defin}[lemma]{Definition}
  \newtheorem{example}[lemma]{Example}
  \newtheorem{remark}[lemma]{Remark}
  \newtheorem{remarks}[lemma]{Remarks}
  \newtheorem{constr}[lemma]{Construction}

  \numberwithin{equation}{section}

  \newcommand{\stacks}[1]{\cite[\href{https://stacks.math.columbia.edu/tag/#1}{Tag~#1}]{StacksProject}}

  \title{Regular logarithmic connections}
  \date{\today}

  \author{Piotr Achinger}
  \address{Institute of Mathematics, Polish Academy of Sciences
    \newline\indent ul. Śniadeckich 8, 00-656 Warsaw, Poland
  }
  \email{pachinger@impan.pl}

\begin{document}

\begin{abstract}
  We introduce the notion of a regular integrable connection on a smooth log scheme over $\CC$ and construct an equivalence between the category of such connections and the category of integrable connections on its analytification, compatible with de Rham cohomology. This extends the work of Deligne (when the log structure is trivial), and combined with the work of Ogus yields a topological description of the category of regular connections in terms of certain constructible sheaves on the Kato--Nakayama space. The key ingredients are the notion of a canonical extension in this context and the existence of good compactifications of log schemes obtained recently by Włodarczyk.
\end{abstract}

\maketitle

{\hypersetup{linkcolor=black} \setlength{\parskip}{0pt}
  \tableofcontents
}
\newpage

\section{Introduction}
\label{s:intro}

The notion of a regular connection was introduced by Deligne \cite{DeligneLNM163} in order to describe the category of complex representations of the topological fundamental group of a smooth scheme $X$ over $\CC$ in an algebraic way. Denote by $\MIC(X/\CC)$ the category of coherent $\cO_X$-modules $E$ endowed with an integrable connection $\nabla\colon E\to E\otimes \Omega^1_{X/\CC}$ (such an $E$ is locally free). We say that $E$ is \emph{regular} if its (multi-valued) holomorphic horizontal sections have moderate growth at infinity. Deligne gave several useful characterizations of regular connections, relating them to the classical notion of Fuchsian singularities of differential equations, and proved that the functor sending $E$ to its analytification $E_{\rm an}$ induces an equivalence of categories
\[ 
  \MIC_{\rm reg}(X/\CC)\isomto \MIC(X_{\rm an}/\CC) \eqno \text{\emph{(Existence Theorem \cite[II 5.9]{DeligneLNM163})}}
\]
between the category of regular connections on $X$ and the category of integrable connections on the analytification $X_{\rm an}$, which in turn is equivalent to the category $\cat{LocSys}_\CC(X_{\rm an})$ of complex local systems on $X_{\rm an}$ via the functor sending $E$ to its sheaf $E^\nabla$ of horizontal sections. The correspondence is compatible with cohomology in the sense that
\[ 
  H^*_{\rm dR}(X, E) \, \simeq \,  H^*_{\rm dR}(X_{\rm an}, E_{\rm an})\, \simeq\, H^*(X_{\rm an}, E_{\rm an}^\nabla) \eqno (\emph{Comparison Theorem \cite[II 6.2]{DeligneLNM163}})
\] 
This equivalence is often called the \emph{Riemann--Hilbert correspondence}, later extended to an equivalence between regular holonomic $D$-modules on $X$ and perverse sheaves on $X_{\rm an}$ by Kashiwara \cite{Kashiwara} and \mbox{Mebkhout} \cite{Mebkhout}. The most important examples of regular connections are those coming from geometry, i.e.\ Gauss--Manin connections or Picard--Fuchs equations. Deligne proved that for a morphism $f\colon Y\to X$ which locally on $X$ admits a relative simple normal crossings compactification and a regular connection $E$ on $Y$, the higher direct images $R^n f_* (E\otimes \Omega^\bullet_{Y/X})$ with the Gauss--Manin connection are regular (\emph{Regularity Theorem} \cite[II 7.9]{DeligneLNM163}). 

A pivotal role in the proofs of the above results is played by the notion of a logarithmic connection. An snc pair $(X, D)$ consists of a smooth scheme $X$ over $\CC$ and a simple normal crossings divisor $D\subseteq X$. An \emph{(integrable) logarithmic connection} on $(X, D)$ is a coherent $\cO_X$-module $E$ together with an integrable connection $E\to E\otimes \Omega_{X/\CC}(\log D)$ where $\Omega_{X/\CC}(\log D)$ is the sheaf of differentials with logarithmic poles along $D$. Such an $E$ might not be locally free in general. As it turns out, the notion of regularity can be phrased in terms of logarithmic connections: for $X$ smooth over $\CC$, an object $E$ of $\MIC(X/\CC)$ is regular if and only if for every open $U\subseteq X$ and every compactification $U\hookrightarrow\ov{U}$ such that $(\ov U, D=\ov{U}\setminus U)$ is an snc pair, the restriction $E|_U$ extends to a logarithmic connection on $(\ov{U}, D)$. Such extensions are not unique, but it is possible to enumerate them, and a~key idea in Deligne's work (\cite[Chap.\ II, Proposition~5.4]{DeligneLNM163}, inspired by Manin \cite{Manin}) is that one can single out a locally free \emph{canonical extension} $\ov E$, the unique one in which the eigenvalues of the residue maps along the components of $D$ (roughly, chosen logarithms of the local monodromy around $D$) belong to the set $\{0\leq {\rm Re}(z)<1\}$. 

Logarithmic connections fit very well into the framework of logarithmic geometry \cite{Kato}. An snc pair $(\underline{X}, D)$ corresponds to the log scheme $X=(\underline{X}, \cM_X)$ where $\cM_X$ is the sheaf of regular functions on $X$ invertible away from $D$; $\underline{X}$ is the ``underlying scheme'' of $X$. Every log scheme $X$  has a sheaf of differentials $\Omega^1_{X/\CC}$, and it is possible to define the category $\MIC(X/\CC)$ of coherent $\cO_X$-modules endowed with a logarithmic connection. For an snc pair $(\underline{X}, D)$, the corresponding log scheme $X$ is smooth in the sense of log geometry, we have $\Omega^1_{X/\CC}\simeq \Omega^1_{\underline{X}}(\log D)$, and integrable connections on $X$ are precisely the integrable logarithmic connections on $(\underline{X}, D)$ as defined in the previous paragraph.

The goal of this article is to extend Deligne's results to smooth logarithmic schemes over $\CC$. Since logarithmic connections are intimately tied to the notion of regularity, it is natural to treat them on the same footing as connections without poles. Given the interest in analytic logarithmic connections \cite{KatoNakayama,IllusieKatoNakayama,OgusRH}, it is a bit surprising that such a theory is has not been developed earlier. That said, my main reason for doing so is the forthcoming work \cite{AchingerRigidRH}, in which I aim to give analogous results for ``complex rigid-analytic spaces,'' with applications to Hodge theory in mind. A typical space of interest in this context is the special fiber of a semistable formal scheme over $\CCs$, endowed with the induced log structure. Since such log schemes do not arise from snc pairs $(\underline{X}, D)$, they are not captured by the classical theory. 

Before going into the details of log geometry, let us explain how our definition of a regular logarithmic connection works for an snc pair $(\underline{X}, D)$. For an integrable logarithmic connection $(E, \nabla)$ on $(\underline{X}, D)$ to be regular, it is insufficient to require that its restriction to $X\setminus D$ be regular (in Deligne's sense), as $E$ might be supported on $D$. To impose conditions along $D$ as well, we proceed as follows. For every stratum $Z$ of the stratification of $X$ induced by the components of $D$, the connection $\nabla$ induces a map $\nabla|_Z\colon E|_Z\to E|_Z\otimes\Omega^1_{X/\CC}(\log D)|_Z$. A local choice of equations of $D$ along $Z$ induces an isomorphism 
\[
  \Omega^1_{X/\CC}(\log D)|_Z\simeq \Omega^1_{Z/\CC}\oplus \cO_Z^{\op{codim} Z},
\]
which allows us to deduce from $\nabla|_Z$ an integrable connection $E|_Z\to E|_Z\otimes\Omega^1_{Z/\CC}$. While this connection depends on the choices made, its being regular does not, and we will say that $E$ is \emph{regular} if these connections are regular, for all $Z$. 

We now turn to stating the general results. In order to gain some extra flexibility, we deal with ``idealized smooth'' log schemes over $\CC$, which look locally like the vanishing set of a monomial ideal in a toric variety (Definition~\ref{def:idealized-smooth}). This class contains log schemes coming from snc as well as special fibers of semistable formal schemes. In \S\ref{ss:splittings}, for such a log scheme $X$, we construct a morphism
\[
  \pi\colon X^\#\ra X
\]
from a smooth scheme $X^\#$ (the disjoint union of certain torsors under tori over the log strata of $X$), and a canonical ``splitting'' $\eps_{\rm univ}$ of the log structure pulled back from $X$. The splitting $\eps_{\rm univ}$ behaves as if it were a section of the map from $X^\#$ to its underlying scheme $\underline{X}^\#$, even though no such section exists, and in particular it allows one to turn log connections into classical connections by means of a ``pull-back'' functor
\[ 
  \epsast_{\rm univ}\colon \MIC(X^\#/\CC) \ra \MIC(\underline{X}^\#/\CC).
\]

\begin{defin}[See Definition~\ref{def:regular}]
  An object $E$ of $\MIC(X/\CC)$ is \emph{regular} if the object $\epsast_{\rm univ}(\pi^* E)$ of $\MIC(\underline{X}^\#/\CC)$ is regular (in Deligne's sense). We denote by $\MIC_{\rm reg}(X/\CC)$ the full subcategory of $\MIC(X/\CC)$ consisting of regular objects. 
\end{defin}

Our notion of regularity relies on Deligne's, which makes many of the proofs easier. We give some familiar-looking characterizations of regularity including a ``cut by curves'' criterion. However, in setting up the theory, some extra difficulties have to be overcome: for example, as the functor $\epsast_{\rm univ}(\pi^* (-))$ is not exact and logarithmic connections might not be locally free, it is not obvious that subobjects of regular objects are regular. It is also not completely trivial to show that every connection on a proper $X$ is regular, as $X^\#$ will typically not be proper. 

The main results of this paper can be summarized as follows.

\begin{thm} \label{thm:main}
  Let $X$ be an idealized smooth log scheme over $\CC$. 
  \begin{enumerate}[(a)]
    \item \emph{(Existence Theorem~\ref{thm:existence})} The analytification functor $E\mapsto E_{\rm an}$ induces an equivalence of categories
    \[ 
      \MIC_{\rm reg}(X/\CC) \isomto \MIC(X_{\rm an}/\CC)
    \]
    \item \emph{(Comparison Theorem~\ref{thm:comparison})} Let $E$ be an object of $\MIC_{\rm reg}(X/\CC)$. Then 
    \[
      H^*_{\rm dR}(X, E)\simeq H^*_{\rm dR}(X_{\rm an}, E_{\rm an}).
    \]
    \item \emph{(Regularity Theorem~\ref{thm:regularity})} Let $f\colon Y\to X$ be a proper, smooth, and exact morphism and let $E$ be an object of $\MIC_{\rm reg}(Y/\CC)$. Assume that Conjecture~\ref{conj:log-ss-reduction} (a form of semistable reduction) holds. Then $R^n f_*(E\otimes\Omega^\bullet_{Y/X})$ are regular  for all $n\geq 0$.
  \end{enumerate}
\end{thm}

In \cite{OgusRH}, Ogus described the category $\MIC(X_{\rm an}/\CC)$ of analytic log integrable connections in terms of topological data, extending the earlier work \cite{IllusieKatoNakayama,KatoNakayama} in the case of (quasi-)unipotent local monodromy.  More precisely, he constructed an equivalence between the category $\MIC(X/\CC)$ of integrable connections on an idealized smooth log complex analytic space $X$ and the category $L(X)$ of certain constructible sheaves on the Kato--Nakayama space $X_{\rm log}$, compatible with cohomology (see Definition~\ref{def:ogus-Lcoh} and Theorem~\ref{thm:ogus-RH}). Combined with Ogus' results, Theorem~\ref{thm:main} shows that we have an equivalence of categories
\[ 
  \MIC_{\rm reg}(X/\CC) \simeq L(X_{\rm an})
\] 
compatible with cohomology. In \S\ref{ss:locsys}, we use it to describe the category $\cat{LocSys}_\CC(X_{\rm log})$ of complex local systems on $X_{\rm log}$ in terms of regular connections on $X$, assuming that stalks of the sheaf $\oM_X$ are free.

As in \cite{DeligneLNM163}, the proof of the Existence Theorem relies on the ability to extend an analytic connection to a compactification. In other words, we need to construct ``canonical extensions'' in the logarithmic context. Suitable compactifications of log schemes have only recently been constructed by Włodarczyk \cite{Wlodarczyk}. We call an open immersion $j\colon X\hookrightarrow \ov{X}$ of log schemes a \emph{good compactification} (Definition~\ref{def:good-embedding}) if $\ov{X}$ is proper and if \'etale locally the map $j$ looks like the inclusion
\[ 
  \Spec(\CC[P\times \ZZ^r]/(K)) \hookrightarrow \Spec(\CC[P\times\NN^r]/(K))
\]
for some $r\geq 0$, a monoid $P$, and an ideal $K\subseteq P$. Then locally, every idealized smooth log scheme over $\CC$ admits a good compactification (Theorem~\ref{thm:good-comp-exists}). The above special form of the embedding (the fact that the ``extra monoids at infinity'' are free) enables us to define the exponents at infinity (analogues of the \emph{negatives} of the eigenvalues of the residue maps) of an object $\ov{E}$ of $\MIC(\ov{X}/\CC)$. We then show (Theorem~\ref{thm:canonical-ext}) that every object of $\MIC(X_{\rm an}/\CC)$ admits a unique extension to an object of $\MIC(\ov{X}/\CC)$ whose exponents at infinity lie in $\{-1< {\rm Re}(z)\leq 0\}$. The proof relies on Ogus' equivalence between $\MIC(X_{\rm an}/\CC)$ and $L(X_{\rm an})$, which allows us to reduce the question to graded modules over monoid algebras.

The definition of a regular logarithmic connection extends naturally to holonomic logarithmic $D$-modules in the sense of Koppensteiner--Talpo \cite{KoppensteinerTalpo,Koppensteiner}, see Remark~\ref{rmk:D-mod-reg-holo}. This direction may deserve further study.

\medskip

\noindent
{\bf $\heartsuit$ Acknowledgments.} The author thanks Adrian Langer, Katharina H\"ubner, Arthur Ogus, Mattia Talpo, Michael Temkin, and Alex Youcis for useful comments and discussions. This work is a part of the project KAPIBARA supported by the funding from the European Research Council (ERC) under the European Union's Horizon 2020 research and innovation programme (grant agreement No 802787).

\section{Preliminaries}
\label{s:prelim}

\subsection{Idealized smooth log schemes}
\label{ss:idealized-smooth}

We recall some terminology regarding (commutative) monoids and log schemes. For a monoid $P$ (written additively), we write $P\to P^\gp$ for its initial homomorphism into a group, and $P^\times\subseteq P$ for the largest subgroup of $P$. We say that $P$ is \emph{integral} if $P\to P^\gp$ is injective. We say that $P$ is \emph{sharp} if $P^\times = 0$ and write $\ov{P}=P/P^\times$ (quotient of the group action of $P^\times$). We say that $P$ is \emph{saturated} if it is integral and for every $x\in P^\gp$ and $n\geq 1$ such that $nx\in {\rm im}(P\to P^\gp)$ we have $x\in {\rm im}(P\to P^\gp)$. A monoid $P$ is is \emph{fs} (fine and saturated) if it is finitely generated and saturated. We note that if $P$ is fs then $\ov{P}^\gp$ is a free abelian group. An \emph{ideal} of a monoid $P$ is a subset $K\subseteq P$ such that $P+K=K$; it is \emph{prime} if $x+y\in K$ implies $x\in K$ or $y\in K$. For an ideal $K$ of $P$ we write $\sqrt{K}$ for the ideal of elements $x\in P$ such that $nx\in K$ for some $n\geq 1$. A \emph{face} of $P$ is a submonoid $F\subseteq P$ such that $x+y\in F$ implies $x,y\in F$ (equivalently, $P\setminus F$ is a prime ideal). For a face $F$ of $P$, we write $P_F = P + F^\gp\subseteq P$ for the localization of $P$ at $F$, and $P/F$ for the quotient $P_F/F^\gp$.

For a log scheme $X$ we write $\underline{X}$ for its underlying scheme and $\alpha\colon \cM_X\to \cO_X$ for its log structure, where $\cM_X$ is written additively. A morphism of log schemes $f\colon X\to Y$ is \emph{strict} if the map of log structures $f^*\cM_Y\to\cM_X$ is an isomorphism. For a ring $R$, a monoid $P$, and a homomorphism $P\to R$ (with addition on $P$ and multiplication on $R$), we write $\Spec(P\to R)$ for the scheme $\Spec(R)$ endowed with the log structure associated to the map $P\to R$. We set $\mathbf{A}_P = \Spec(P\to \ZZ[P])$; later on, when we work over $\CC$, we shall write $\mathbf{A}_P$ for $\Spec(P\to \CC[P])$. A log scheme $X$ is \emph{fs} if \'etale locally on $\underline{X}$ there exists an fs monoid $P$ and a strict morphism $X\to \mathbf{A}_P$. A morphism of log schemes $Y\to X$ is \emph{strict \'etale} if it is both strict and \'etale, or equivalently if it is strict and the underlying morphism of schemes $\underline{Y}\to\underline{X}$ is \'etale. 

\begin{defin} \label{def:idealized-smooth} 
  \begin{enumerate}[(a)]
    \item Let $P$ be an fs monoid and let $K\subseteq P$ be an ideal. We set
    \[ 
      \mathbf{A}_{P,K} = \Spec(P\to \CC[P]/(K))
    \] 
    for the scheme cut out by $K$ in $\mathbf{A}_P = \Spec(P\to \CC[P])$ endowed with the induced log structure.
    \item Let $X$ be an fs log scheme over $\CC$. We say that $X$ is \emph{idealized smooth} if strict \'etale locally on $X$ there exist an fs monoid $P$, an ideal $K\subseteq P$, and a strict \'etale morphism $X\to \mathbf{A}_{P,K}$.
  \end{enumerate}
\end{defin}

We make the analogous definition for fs log complex analytic spaces, using ``locally'' instead of ``\'etale locally'. Then if $X$ is an idealized smooth log scheme over $\CC$, then its analytification $X_{\rm an}$ is an idealized smooth log complex analytic space.

The following lemma often allows one to reduce to the case when the ideal $K$ is trivial.

\begin{lemma} \label{lem:local-embedding}
  Let $X$ be an idealized smooth log scheme over $\CC$. Then strict \'etale locally on $X$ there exist an fs monoid $P$, an ideal $K\subseteq P$ and a cartesian square
  \[
    \xymatrix{
      X \ar[d] \ar[r] & Y \ar[d] \\
      \mathbf{A}_{P,K} \ar[r] & \mathbf{A}_P
    }
  \]
  where $\mathbf{A}_{P,K}\to \mathbf{A}_P$ is the natural strict closed immersion and where the vertical maps are strict \'etale. In particular, $Y$ is smooth over $\CC$ and $X$ is cut out by a coherent ideal in $\cM_Y$.
\end{lemma}

\begin{proof}
We may assume that $X$ is affine and that it admits a strict \'etale map $X\to \mathbf{A}_{P,K}$ for $P$ and $K$ as in the statement. Write $A=\CC[P]$, $I=K\cdot A$, $\ov{A} = A/I$, and $X=\Spec(\ov{B})$. Thus $\ov{B}$ is an \'etale $\ov{A}$-algebra. By Artin approximation \cite[Th\'eor\`eme 3]{Elkik} (which in the case of \'etale morphisms can be proved directly lifting the presentation in \stacks{03PC}(3)), there exists an \'etale $A$-algebra $A'$ with $A'/IA' \simeq \ov{A}$ and an \'etale $A'$-algebra $B'$ such that $\ov{B} \simeq B'/IB'$ over $A$. Set $Y = \Spec(B')$ and endow it with the log structure induced by the \'etale map $Y\to \Spec(A) = \mathbf{A}_P$. 
\end{proof}

\subsection{Compactifications of log schemes}
\label{ss:compactifications}

\begin{defin}[Good embedding] \label{def:good-embedding}
  Let $X$ be an idealized smooth log scheme over $\CC$ or an idealized smooth log complex analytic space.
  \begin{enumerate}[(a)]
    \item A \emph{good embedding} is an open immersion $j\colon X\hookrightarrow \ov{X}$ which is locally isomorphic to the inclusion
    \[ 
      \mathbf{A}_{P, K}\times \mathbf{G}_m^r \hookrightarrow \mathbf{A}_{P, K} \times \mathbf{A}^r 
    \]
    for some $r\geq 0$, some fs monoid $P$, and some ideal $K\subseteq P$. 

    \item For a good embedding, we denote by $\mathcal{F}$ the kernel of $\ov{\cM}_{\ov{X}}\to j_* \ov{\cM}_X$. It is a sheaf of faces of $\ov{\cM}_{\ov{X}}$ whose stalks are free monoids, which we call the \emph{residual face} of the good embedding $X\hookrightarrow \ov{X}$.

    \item A \emph{good compactification} of $X$ is a good embedding $j\colon X\hookrightarrow \ov{X}$ where $X$ is proper.
  \end{enumerate}
\end{defin}

Note that if $j\colon X\to \ov{X}$ is a good compactification, then $\cM_{\ov X} = j_* \cM_X$. Indeed, assuming $K=\emptyset$, the above description shows that the log structure is everywhere nontrivial outside $X$, i.e.\ we have $\ov{X}{}^\circ = X^\circ$ where $X^\circ$ denotes the largest open subset on which the log structure is trivial. Therefore
\[ 
  \cM_{\ov X} = \bar u_* \cO_{\ov{X}{}^\circ}^\times = \bar u_* \cO_{X^\circ}^\times = j_* u_* \cO_{X^\circ}^\times = j_* \cM_X,
\]
where $u\colon X^\circ\to X$ and $\bar u\colon \ov{X}{}^\circ \to \ov{X}$ are the inclusions.  

\begin{thm} \label{thm:good-comp-exists}
  Let $X$ be an idealized smooth log scheme over $\CC$. Then there exists a strict \'etale cover $\{X_i\to X\}_{i\in I}$ and good compactifications $X_i\hookrightarrow \ov{X}_i$.
\end{thm}

\begin{proof}
Assume first that $X$ is smooth. We phrase the proof in the language of toroidal embeddings used in \cite{Wlodarczyk}: let $D\subseteq X$ be the divisor where the log structure of $X$ is nontrivial. The question being \'etale local, we may assume that $X$ is affine, that the toroidal embedding $(X, D)$ is strict (i.e., the log structure admits a chart Zariski locally), and that $D$ has a unique closed stratum. By \cite[Example~2.2.7]{Wlodarczyk}, such a toroidal embedding is extendable (\cite[Definition~2.2.2]{Wlodarczyk}). By \cite[Theorem~2.2.14]{Wlodarczyk}, there exists a compactification $j\colon X\to \ov{X}$ such that $(\ov{X}, \ov{D})$ is again a strict toroidal embedding, where $\ov{D}$ denotes the closure of $D$ in $\ov{X}$, and such that the divisor $E = \ov{X}\setminus X$ has \emph{simple normal crossings} with $\ov{D}$ (\cite[Definitions~2.1.13 and 2.1.15]{Wlodarczyk}). The latter means precisely that $(\ov{X}, \ov{D}+E)$ is a strict toroidal embedding which locally has the form 
\[
  (\mathbf{A}_P\times \mathbf{A}^r, (\partial \mathbf{A}_P) \times \mathbf{A}^r + \mathbf{A}_P \times (\partial \mathbf{A}^r)),
\]
where $\partial \mathbf{A}_P$ (resp.\ $\partial \mathbf{A}^r$) denotes the complement of $\mathbf{A}_{P^{\rm gp}}$ (resp.\ $\mathbf{G}_m^r$). This means that the log scheme corresponding to the toroidal embedding $(\ov{X}, \ov{D}+E)$ is a good compactification of $X$.

In the general case, by Lemma~\ref{lem:local-embedding} we may assume that $X$ admits an exact closed immersion $X\hookrightarrow Y$ where $Y$ is smooth and $X$ is cut out by a coherent ideal in $\cM_Y$. Localizing further, we may assume that $Y$ admits a good compactification $Y\hookrightarrow \ov{Y}$. Let $\ov{X}$ be the scheme-theoretic closure of $X$ in $\ov{Y}$, endowed with the log structure induced from $\ov{Y}$. Then $X\hookrightarrow \ov{X}$ is a good compactification as well. To see this, it is enough to observe that the scheme-theoretic closure of $\mathbf{A}_{P,K}\times \mathbf{G}_m^r$ in $\mathbf{A}_P\times \mathbf{A}^r$ equals $\mathbf{A}_{P,K}\times \mathbf{A}^r$.
\end{proof}

\subsection{Splittings of log structures} 
\label{ss:splittings}

The basic invariant of an fs log structure $\cM_X$ on a scheme $X$ is the constructible sheaf of monoids $\oM_X = \cM_X/\cO_X^\times$, giving rise to the \emph{log stratification} of $X$ on whose strata it is locally constant. In this section, we deal with the situation when $\oM_X$ is locally constant, in which case the log structure on $X$ is ``locally split.'' 

The main takeaway point for the rest of the paper is that for an idealized smooth log scheme $X$ over $\CC$, there exists a strict morphism $X^\#\to X$ where $X^\#$ admits a splitting of the log structure, universal for all strict maps $Y\to X$ with $Y$ reduced. Moreover, the underlying scheme of $X^\#$ is smooth over $\CC$.

We treat the case of fs log schemes, but the assertions apply equally well to fs log complex analytic spaces, formal schemes, algebraic spaces \&c., with the meaning of the word \emph{locally} adjusted accordingly.

\begin{defin} \label{def:locconst-hollow-split}
  Let $X$ be an fs log scheme.
  \begin{enumerate}[(a)]
    \item We say that $X$ has \emph{locally constant log structure} if the sheaf $\oM_X$ is locally constant. 
    \item We say that $X$ is \emph{hollow} \cite[\S 2, Definition~4]{OgusHK} if $\alpha(\cM_{X,x}\setminus \cO_{X,x}^\times)=0$ for every geometric point $x$.
    \item A \emph{splitting} of the log structure $\cM_X$ is a homomorphism of sheaves of monoids $\eps\colon \oM_X\to\cM_X$ such that the composition $\oM_X\to\cM_X\to\oM_X$ is the identity.
  \end{enumerate}
\end{defin}

For an integral monoid $P$, splittings of the projection $P\to \ov{P}=P/P^*$ are in bijection with splittings of the associated map of groups $P^\gp\to \ov{P}{}^\gp$. Therefore, if $X$ is an fs log scheme which locally admits a splitting and such that $\oM_X$ is locally constant (as we shall shortly see, the two assumptions turn out to be equivalent), then splittings on $X$ correspond to splittings of the locally split exact sequence of sheaves
\[ 
  \xymatrix{ 1\ar[r] & \cO_X^\times \ar[r] & \cM^{\rm gp}_X \ar[r] & \oM^\gp_X\ar[r] & 0}
\]
and therefore form a torsor under $T_X = \underline{\Hom}(\oM^\gp_X, \cO_X^\times)$. The sheaf $T_X$ is the sheaf of sections of the torus over $X$ with character sheaf $\oM_X^{\rm gp}$ (see \cite[Expos\'e X, \S 5--7]{SGA3}).

\begin{lemma} \label{lem:splitting-basics}
  Let $X$ be an fs log scheme.
  \begin{enumerate}[(a)]
    \item Let $\eps$ be a splitting on $X$, and let $m$ be a local section of $\oM_X$. Then $\alpha(\eps(m))$ is locally on $X$ either a unit or nilpotent.\footnote{In the case of formal schemes, replace \emph{nilpotent} with \emph{topologically nilpotent}. Indeed, in an adic ring the intersection of all open prime ideals is the set of topologically nilpotent elements.}
    \item $X$ has locally constant log structure if and only if $\cM_X$ locally admits a splitting.
    \item If $X$ is hollow, then $X$ has locally constant log structure. Conversely, if $X$ is reduced and has locally constant log structure, then $X$ is hollow.
    \item Let $P$ be an fs monoid and let $K\subseteq P$ be an ideal. Then $\mathbf{A}_{P,K}$ has locally constant log structure if and only if $\sqrt{K} = P\setminus P^\times$, i.e.\ if every element $x$ of $P$ is either invertible or satisfies $nx \in K$ for some $n\geq 1$. It is hollow if and only if $K = P\setminus P^\times$, and in this case we may write $P = Q\oplus M$ with $M$ a group and $Q$ a monoid such that $\mathbf{A}_{P,K} \simeq \Spec(Q\to \CC[M])$ where the map sends $Q^\times$ to $1$ and $Q\setminus Q^\times$ to $0$.
  \end{enumerate}
\end{lemma}

\begin{proof}
(a) We may assume $X=\Spec(R)$ for a strictly henselian local ring $R$, with closed point $x$. Let $s\in \cM_{X,x}$ and let $f=\alpha(s) \in \cO_{X,x}=R$. Since nilpotent elements are the intersection of all prime ideals, we need to show that if $f\notin \mathfrak{p}$ for some prime ideal $\mathfrak{p}\subseteq R$, then $f$ is a unit. Let $y\in X$ be the point corresponding to such a $\mathfrak{p}$. We look at the commutative square
\[ 
  \xymatrix{
    \oM_{X, x} \ar[d]_\eps \ar[r] & \oM_{X,y} \ar[d]^\eps  \\
    \cM_{X,x} \ar[r] & \cM_{X,y}.
  }
\]
Let $\ov{s}\in \oM_{X,x}$ be the image of $s$. We have $\eps(\ov{s}) = us$ for some $u\in \cO^\times_{X,x}=R^\times$. As $\alpha(s)=f\notin \mathfrak{p}$, i.e.\ the image of $\alpha(s)$ in $\cO_{X,y}$ is a unit, the image of $s$ in $\cM_{X,y}$ is in $\cO^\times_{X,y}$, so $\ov{s}$ maps to $0$ in $\oM_{X,y}$. Therefore $\eps(\ov{s})=us$ maps to $1$ in $\cM_{X,y}$, and $uf$ maps to $1$ in $\cO_{X,y}$. But then $uf-1$ maps to zero in $\cO_{X,y}$, and hence is non-invertible in $\cO_{X,x}$, so $f$ is invertible.

(b) Suppose that $X$ is locally split. By the proof of (a), we see that if $x$ specializes to $y$, then $\oM_{X,x}\to \oM_{X,y}$ is bijective. By \cite[\S 2, Lemma~5]{OgusHK}, $\oM_X$ is locally constant. Conversely, suppose that $\oM_X$ is locally constant, and we want to produce a splitting locally on $X$. To this end, we pick a point $x$ and we may assume that $\oM_X$ is constant with value $P$ and $X$ admits a neat chart \cite[Chapter II, Definition 2.3.1]{OgusBook} at $x$, i.e.\ a chart $P\to\Gamma(X, \cM_X)$. This chart defines a splitting $\oM_{X}\to \cM_X$.

(c) Follows from (a) and (b).

(d) Omitted.
\end{proof}

Note that splittings of the log structure can be pulled back along strict morphisms. That is, if $f\colon Y\to X$ is a strict morphism of fs log schemes and $\eps$ is a splitting of $\cM_X$, we have a natural map  
\[ 
  f^*(\eps) \colon \oM_Y = f^* \oM_X \ra (\text{pullback under $f$ of $\cM_X$ as a sheaf}) \ra f^*\cM_X = \cM_Y
\]
which is a splitting of $\cM_Y = f^*\cM_X$. This construction gives rise to the following functors $\cat{Sch}_{\underline{X}}^{\rm op} \to \cat{Sets}$: 
\begin{align*}
  \widehat{X}^\#(Y\xto{f} X) &= \{ \text{splittings of $f^*\cM_X$} \}, \\
  \widehat{X}^\flat(Y\xto{f} X) &= 
    \begin{cases}
      \{\ast\} & \text{if $f^*\cM_X$ is locally constant} \\
      \emptyset & \text{otherwise}
    \end{cases} \\
  X^\#(Y\xto{f} X) 
    &= \begin{cases}
      \{ \text{splittings of $f^*\cM_X$} \} & \text{if $(Y, f^*\cM_X)$ is hollow} \\
      \emptyset & \text{otherwise}
    \end{cases} \\ 
  X^\flat(Y\xto{f} X) &= 
    \begin{cases}
      \{\ast\} & \text{if $(Y, f^*\cM_X)$ is hollow} \\
      \emptyset & \text{otherwise}
    \end{cases} 
\end{align*}
We have natural transformations forming a cartesian square of functors
\[ 
  \xymatrix{
    X^\# \ar[r] \ar[d] & \widehat{X}^\# \ar[d] \\
    X^\flat \ar[r] & \widehat{X}^\flat.
  }
\]
We shall shortly recognize $X^\flat$ as the ``log stratification'' of $X$. The space $X^\#$, forming a torus torsor over $X^\flat$, will play a key role in this paper. The spaces $\widehat{X}^\#$ and $\widehat{X}^\flat$ will not be used in what follows.

\begin{remark} \label{rmk:Xsplit-functorial}
Let $X$ be a scheme and let $\mathcal{N}\to \cM$ be a morphism of log structures on $X$. We get a commutative diagram of sheaves of abelian groups
\[ 
  \xymatrix{
    1\ar[r] & \cO_X^\times \ar@{=}[d] \ar[r] & \mathcal{N}^\gp \ar[r] \ar[d] & \ov{\mathcal{N}}^\gp\ar[r]\ar[d] & 0 \\
    1\ar[r] & \cO_X^\times \ar[r] & \cM^\gp \ar[r] & \oM^\gp \ar[r] & 0. 
  }
\]
Therefore the right square is a pullback, and consequently for every splitting $\eps$ of $\cM$ there exists a unique splitting $\eps'$ of $\mathcal{N}$ for which the square
\[ 
  \xymatrix{
    \ov{\mathcal{N}} \ar[r]^{\eps'} \ar[d] & \mathcal{N} \ar[d] \\
    \oM \ar[r]_\eps & \cM 
  }
\]
commutes. If $f\colon Y\to X$ is a map of log schemes and $\eps_X$ and $\eps_Y$ are splittings on $X$ and $Y$ respectively, we shall say that $f$ is \emph{compatible} with $\eps_X$ and $\eps_Y$ if the splitting of $f^*\cM_X$ induced by $\eps_X$ by pullback is equal to the splitting of $f^*\cM_X$ induced by $\eps_Y$ as above.

This discussion implies that the four functors we have introduced are functorial in the sense that a map $Y\to X$ of fs log schemes induces maps $Y^\#\to X^\#$, $Y^\flat\to X^\flat$ etc. Moreover, the map $Y^\#\to X^\#$ is compatible with the universal splittings.
\end{remark}

\begin{prop} \label{prop:splitting-functors}
  Let $X$ be an fs log scheme. 
  \begin{enumerate}[(a)]
    \item The functors $\widehat{X}^\#$ and $\widehat{X}^\flat$ are representable by formal schemes affine over $X$, and the map $\widehat{X}^\#\to \widehat{X}^\flat$ is smooth and affine.
    \item The functors $X^\#$ and $X^\flat$ are representable by schemes of finite presentation and affine over $X$. The map $X^\#\to X^\flat$ is smooth and affine (a torsor under a torus over $X^\flat$) and the map $X^\flat\to X$ is a universally bijective monomorphism (the disjoint union of a locally finite family of finitely presented locally closed subschemes of $X$).
  \end{enumerate}
\end{prop}

\begin{example} \label{ex:A1split}
Let $X = \mathbf{A}^1 = \Spec(\NN\xto{1\mapsto t} \CC[t])$. Then
\begin{align*}
  \widehat{X}^\flat & \quad\simeq\quad \Spec(\CC[t,t^{-1}]) \quad\sqcup \quad\op{Spf}(\CC\llbracket t\rrbracket), \\
  \widehat{X}^\# & \quad\simeq\quad \Spec(\CC[t,t^{-1}]) \quad\sqcup\quad \op{Spf}(\CC\llbracket t\rrbracket\langle y, y^{-1}\rangle), \\
  X^\flat & \quad\simeq\quad \Spec(\CC[t, t^{-1}]) \quad\sqcup\quad \Spec(\CC), \\
  X^\# & \quad\simeq\quad \Spec(\CC[t, t^{-1}]) \quad\sqcup\quad \Spec(\CC[y, y^{-1}]). 
\end{align*}
Here, $\Spec(A)$ is treated as the formal spectrum of $A$ equipped with the discrete topology, and $\CC\llbracket t\rrbracket\langle y, y^{-1}\rangle$ is the $t$-adic completion of $\CC[t, y, y^{-1}]$.
\end{example}

\begin{example} \label{ex:APsplit}
Let $X = \mathbf{A}_{P}$ for an fs monoid $P$. For a face $F$ of $P$, we denote by $X_F\hookrightarrow X$ the strict locally closed immersion corresponding to the map
\begin{equation} \label{eqn:P-Fgp-map}
  \CC[P] \ra \CC[F^{\rm gp}], \qquad
  p \mapsto \begin{cases}
    p & \text{if $p\in F$}, \\
    0 & \text{otherwise}.
  \end{cases}
\end{equation}
Thus $\underline{X}_F \simeq \mathbf{A}_{F^\gp}$. As we shall prove shortly, we have
\begin{align}
  X^\# &\simeq \coprod_F X_F\times \mathbf{A}_{(P/F)^{\rm gp}}, \label{eqn:Xsplit-formula}  \\
  X^\flat &\simeq \coprod_F X_F, \label{eqn:Xstrat-formula}
\end{align}
the coproducts taken over all faces $F\subseteq P$. The map $X^\#\to X^\flat$ is the disjoint union of the projections $X_F\times \mathbf{A}_{(P/F)^{\rm gp}}\to X_F$.
Note that as $P/F$ is sharp, $(P/F)^\gp$ is a free abelian group, so the short exact sequence $0\to F^\gp\to P^\gp\to (P/F)^\gp\to 0$ splits. Therefore $\underline{X}^\#$ is non-canonically isomorphic to the disjoint union of copies of the torus $X^* = \mathbf{A}_{P^\gp}$ indexed by faces of the monoid $P$. However, the log structures on the components depend on $F$.
\end{example}

\begin{remark}
The scheme $X^\#$ is a toy model of the Kato--Nakayama space $X_{\rm log}$ of $X$. Its class in the Grothendieck ring of varieties $K_0(\cat{Var}_\CC)$ seems to be the ``correct'' class associated to the log scheme~$X$.
\end{remark}

\begin{proof}[Proof of Proposition~\ref{prop:splitting-functors}]
We first deal with the case $X = \mathbf{A}_P = \Spec(P\to\ZZ[P])$ and reduce to this case afterwards. We will use the notation of Example~\ref{ex:APsplit}.

We handle $\widehat{X}^\flat$ first. For a face $F$ of $P$, let $P_F = P - F \subseteq P^\gp$ be the localization of $P$ at $F$, and let $K_F = P_F \setminus F^\gp$, which generates the kernel of \eqref{eqn:P-Fgp-map}. Let $\mathfrak{X}_{F}$ denote the formal spectrum of the $K_F$-adic completion of $\ZZ[P_F]$. We claim that $\widehat{X}^\flat$ is represented by $\coprod_F \mathfrak{X}_{F}$. Since on the image of $X_F\to X$ the sheaf $\oM_X$ is constant with value $P/F$, the log structure on each $\mathfrak{X}_{F}$ is locally constant. Let $Y\to X$ be a map such that the induced log structure on $Y$ is locally constant. In particular, locally on $Y$, the pull-back of every $p\in P$ is either nilpotent or a unit (Lemma~\ref{lem:splitting-basics}(a)). Since $P$ is finitely generated, working locally on $Y$, we may assume that every $p$ is either nilpotent or a unit or $Y$. Let $F\subseteq P$ be the set of $p\in P$ such that $p$ is a unit on $Y$. Clearly, $F$ is a face of $P$, the map $P\to \Gamma(Y, \cO_Y)$ factors through $P_F$ and sends a power of $K_F$ to zero. Therefore $Y\to X$ factors uniquely through $\mathfrak{X}_{F}$.

The same proof shows that if $Y\to\mathbf{A}_P$ is such that the induced log structure on $Y$ is hollow, then the nilpotent functions above have to vanish, forcing the map to factor through $\Spec(\ZZ[P_F]/(K_F)) = X_F$. Therefore $X^\flat$ is represented by the disjoint union of $X_F$ taken over all faces $F$ of $P$, and we proved \eqref{eqn:Xstrat-formula}. 

Turning to $X^\#$ and $\widehat{X}^\#$, we need to show that if $X$ has locally constant log structure to begin with, then the functor $\widehat{X}^\#$ is representable by a smooth scheme over $X$. But, the discussion preceding Lemma~\ref{lem:splitting-basics} shows that it is a torsor under the torus $\underline{\Hom}(\oM^\gp_X, \mathbf{G}_m)$ and hence is representable (e.g.\ by \cite[Theorem~4.3(a)]{MilneEtale}). A more direct analysis shows that in the case $X = \Spec(P\to \ZZ[F^\gp])$ as in Example~\ref{ex:APsplit}, we have $\Gamma(X, \cM_X) \simeq (P/F)\times F^\gp$. This gives the isomorphism \eqref{eqn:Xsplit-formula}.

Finally, we deal with the general case. Suppose first that $X'\to X$ is a strict \'etale surjection and that the statement holds for $X'$. Since the formation of the four functors $\widehat{X}^\#$, $\widehat{X}^\flat$, $X^\#$, $X^\flat$ commutes with strict base change, we deduce that they are representable by (formal) algebraic spaces over $X$ with the required properties, and it remains to show that they are in fact (formal) schemes. However, $X^\flat\to X$ is affine, and hence $X^\flat$ is a scheme by \'etale descent for quasi-coherent sheaves. Similarly, $\widehat{X}^\flat$ is a formal scheme. Finally, $X^\#\to X^\flat$ and $\widehat{X}^\#\to \widehat{X}^\flat$ are affine as well, so $X^\#$ is a scheme and $\widehat{X}^\#$ is a formal scheme. Therefore the statement holds for $X$.  

By the above, we may reduce to the case when $X$ admits a strict morphism $X\to \mathbf{A}_{P}$ for some monoid $P$, and since the statement holds for $\mathbf{A}_P$ and the four functors commute with base change, the statement holds for $X$ as well.
\end{proof}

Although $X^\#$ and $X^\flat$ were defined (or described) as schemes over $X$, we endow them with the log structures pulled back from $X$. We denote by $\eps_{\rm univ}$ the canonical splitting of the log structure on $X^\#$. 

\begin{example} \label{ex:AP-univ-splitting}
Let $X = \Spec(P\to \CC)$ be the log point associated to the sharp fs monoid $P$. Then $X^\# = \Spec(P\to \CC[P^\gp])$ where $P\setminus P^\times$ maps to zero. This chart corresponds to an ``obvious'' splitting $\eps$, and $\Gamma(X^\#, \cM_{X^\#})\simeq P\times P^\gp$ with $\eps$ mapping $p$ to $(p, 0)$. This is not the same as the universal splitting $\eps_{\rm univ}$, which maps $p$ to $(p, p)$. 
\end{example}

Using the local picture (Example~\ref{ex:APsplit}) we deduce the following.

\begin{cor} \label{cor:Xsplit-smooth}
  Let $X$ be an idealized smooth log scheme over $\CC$. Then, the underlying schemes of $X^\#$ and $X^\flat$ are smooth schemes over $\CC$. 
\end{cor}

The following definition will be used in \S\ref{ss:regularity}.

\begin{defin} \label{def:log-dominant}
  A quasi-compact morphism of fs log schemes $f\colon Y\to X$ is \emph{log dominant} if the following conditions are satisfied:
  \begin{enumerate}
    \item $f$ is log injective, i.e.\ the maps $\oM_{X,f(\ov y)}\to \oM_{Y,\ov y}$ are injective for all geometric points $\ov{y}\to Y$,
    \item the restriction of $f$ to every log stratum of $X$ is dominant, or equivalently, if the map $Y^\flat\to X^\flat$ is dominant.
  \end{enumerate}
\end{defin}

We note that for a log dominant morphism $Y\to X$, the induced morphism $Y^\#\to X^\#$ is a dominant morphism of schemes.

\subsection{Algebraic and analytic sections}
\label{ss:GAGA}

Let $X$ be a scheme locally of finite type over $\CC$, let $X_{\rm an}$ be the associated complex analytic space, and let $u\colon X_{\rm an}\to X$ be the canonical map of locally ringed spaces. For a quasi-coherent $\cO_X$-module $\mathcal{F}$, we write $\mathcal{F}_{\rm an}$ its pull-back $u^* \mathcal{F}$ to $X_{\rm an}$. The functor $\mathcal{F}\mapsto \mathcal{F}_{\rm an}$ is exact, and $\mathcal{F}_{\rm an}$ is a coherent analytic sheaf if $\mathcal{F}$ is coherent. We say that an analytic section $s\in \Gamma(X_{\rm an}, \mathcal{F}_{\rm an})$ is \emph{algebraic} (see \cite[II.2]{DeligneLNM163}) if it lies in the image of the (injective) map
\[ 
  u^* \colon \Gamma(X, \mathcal{F}) \ra \Gamma(X_{\rm an}, \mathcal{F}_{\rm an}).
\]
The following result establishes a useful criterion for algebraicity of analytic sections of coherent sheaves.

\begin{prop} \label{prop:analytic-sections}
  Let $X$ be a scheme locally of finite type over $\CC$ and let $\mathcal{F}$ be a coherent $\cO_X$-module. Let $\{Z_i\}_{i\in I}$ be a locally finite family of locally closed subschemes of $X$ such that the union of $|Z_i|$ equals $|X|$. For $n\geq 0$, we denote by $Z_i{}^{(n)}$ the $n$-th order thickening of $Z_i$ in $X$. Let $s\in \Gamma(X_{\rm an}, \mathcal{F}_{\rm an})$ be an analytic section of $\mathcal{F}$. If for every $i\in I$ and every $n\geq 0$, the restriction of $s$ to $Z_i{}^{(n)}$ is algebraic, then so is $s$.
\end{prop}

In other words, if we set
\[
  \pi\colon X' = \coprod_{i\in I} \coprod_{n\geq 0} Z_i^{(n)}\ra X,
\]
the proposition asserts that the following square is cartesian.
\[ 
  \xymatrix{
    \Gamma(X, \mathcal{F}) \ar[r]^-{u^*} \ar[d]_{\pi^*} & \Gamma(X_{\rm an}, \mathcal{F}_{\rm an}) \ar[d]^{\pi_{\rm an}^*} \\
    \Gamma(X', \pi^*\mathcal{F}) \ar[r]_-{u^*} & \Gamma(X'_{\rm an}, (\pi^* \mathcal{F})_{\rm an}). 
  }
\]

\begin{remark} \label{rmk:xexpy}
As the example $X=\Spec(\CC[x,y]/(y^2))$, $Z=V(y)$, $\mathcal{F}=\cO_X$, $s=y\exp(x)$ shows, it is not enough to consider the restrictions to $Z_i$ and not to their thickenings. The issue is that $\coprod Z_i\to X$ is not flat (see Corollary~\ref{cor:associated-pts} below). Of course neither is $X'\to X$, but effectively $X'$ approximates of the formal scheme $\mathfrak{X}'$ which is the disjoint union of the formal completions $\mathfrak{Z}_i$ of $X$ along $Z_i$, and $\mathfrak{X}'\to X$ is flat. Since we do not wish to make sense of $\mathfrak{X}'_{\rm an}$ (a formal complex analytic space), we use the scheme $X'$ instead of the formal scheme $\mathfrak{X}'$. The flatness of $\mathfrak{X}'\to X$ is reflected in the use of the Artin--Rees lemma in the proof below.
\end{remark}

\begin{lemma} \label{lemma:gaga-inj}
  Let $X$ be a scheme locally of finite type over $\CC$ and let $\mathcal{F}'\hookrightarrow \mathcal{F}$ be an injection of quasi-coherent $\cO_X$-modules. A section $s\in \Gamma(X_{\rm an}, \mathcal{F}'_{\rm an})$ is algebraic if and only if its image in $\Gamma(X_{\rm an}, \mathcal{F}_{\rm an})$ is.
\end{lemma}

\begin{proof}
The assertion to be proved is equivalent to the left square in the diagram below being cartesian.
\[ 
  \xymatrix{
    0\ar[r] & \Gamma(X, \mathcal{F}') \ar[r] \ar@{^{(}->}[d]_{u^*} & \Gamma(X, \mathcal{F}) \ar[r] \ar@{^{(}->}[d]_{u^*} & \Gamma(X, \mathcal{F}/\mathcal{F}') \ar@{^{(}->}[d]_{u^*} \\ 
    0\ar[r] & \Gamma(X_{\rm an}, \mathcal{F}_{\rm an}') \ar[r] & \Gamma(X_{\rm an}, \mathcal{F}_{\rm an}) \ar[r] & \Gamma(X_{\rm an}, \mathcal{F}_{\rm an}/\mathcal{F}_{\rm an}') 
  }
\]
The rows being exact and the right vertical map being injective, the claim follows by a simple diagram chase.
\end{proof}

\begin{cor} \label{cor:associated-pts}
  Let $X$ be a scheme locally of finite type over $\CC$ and let $\mathcal{F}$ be a coherent $\cO_X$-module. Let $f\colon X'\to X$ be a flat morphism of finite type whose image contains the associated points of $\mathcal{F}$. Then, a section $s\in \Gamma(X_{\rm an}, \mathcal{F}_{\rm an})$ is algebraic if and only if its image in $\Gamma(X'_{\rm an}, f_{\rm an}^* \mathcal{F}_{\rm an})$ is.
\end{cor}

\begin{proof}
Under the given assumptions, the map $\mathcal{F}\hookrightarrow f_* f^* \mathcal{F}$ is injective. We have a map $\alpha\colon (f_* f^* \mathcal{F})_{\rm an}\to f_{\rm an, *}(f^* \mathcal{F})_{\rm an}$ obtained by adjunction from the analytification of the counit map $f^* f_* f^* \mathcal{F}\to f^*\mathcal{F}$. It features inside a commutative diagram
\[ 
  \xymatrix{
    \Gamma(X, \mathcal{F}) \ar[r] \ar[d] & \Gamma(X, f_* f^* \mathcal{F}) \ar@{=}[rr] \ar[d] & & \Gamma(X', f^*\mathcal{F}) \ar[d] \\
    \Gamma(X_{\rm an}, \mathcal{F}_{\rm an}) \ar[r] & \Gamma(X_{\rm an}, (f_* f^* \mathcal{F})_{\rm an}) \ar[r]_-\alpha & \Gamma(X_{\rm an}, f_{\rm an, *}(f^*\mathcal{F})_{\rm an}) \ar@{=}[r] & \Gamma(X_{\rm an}', (f^*\mathcal{F})_{\rm an} )
  }
\]
Since the map denoted $\alpha$ in the above diagram is injective, the image of $s$ in $\Gamma(X_{\rm an}, (f_* f^* \mathcal{F})_{\rm an})$ is algebraic. It remains to invoke Lemma~\ref{lemma:gaga-inj}.
\end{proof}

\begin{proof}[Proof of Proposition~\ref{prop:analytic-sections}]
The question is local, so we can assume that $X$ is affine and that there exists a finite filtration 
\[ 
  0 = \mathcal{F}_0 \subseteq \mathcal{F}_1 \subseteq \cdots \subseteq \mathcal{F}_m = \mathcal{F}
\]
of $\mathcal{F}$ by coherent subsheaves whose graded pieces are of the form
\[ 
  \mathcal{F}_i / \mathcal{F}_{i-1} \simeq \cO_{Y_i}
\]
for some integral closed subschemes $Y_i\subseteq X$ \cite[Chap.\ 7, Ex.\ 18]{AtiyahMacdonald}. We will proceed by induction on the length $m$ of this filtration.

Consider $m=1$, i.e.\ $\mathcal{F} \simeq \cO_Y$ for some closed integral subscheme $Y\subseteq X$. There exists an index $i$ such that $Z_i\cap Y$ is a dense open subset of $Y$. We conclude by applying Corollary~\ref{cor:associated-pts} to the open immersion $Z_i\cap Y\hookrightarrow Y$ and the sheaf $\cO_Y$.

For the induction step, it suffices to prove that if 
\[
  \xymatrix{0\ar[r] & \mathcal{F}' \ar[r] & \mathcal{F} \ar[r] & \mathcal{F}'' \ar[r] & 0}
\]
is a short exact sequence of coherent $\cO_X$-modules such that the assertion of the proposition holds for both $\mathcal{F}'$ and $\mathcal{F}''$, then it holds also for $\mathcal{F}$. Let $s''\in \Gamma(X_{\rm an},\mathcal{F}''_{\rm an})$ be the image of $s$. Then $\pi^*_{\rm an} s''$ is algebraic, and by the assumption on $\mathcal{F}''$ so is $s''$. Since $X$ is affine, the short exact sequence stays exact after taking global sections, and hence there exists an $s_0 \in \Gamma(X, \mathcal{F})$ whose image in $\Gamma(X, \mathcal{F}'')$ equals $s''$. 

Let $s' = s-s_0 \in \Gamma(X_{\rm an}, \mathcal{F}'_{\rm an})$. Suppose that $\pi^*_{\rm an} s'$ is algebraic, then by the assumption on $\mathcal{F}'$ so is $s'$. Therefore $s = s' + s_0$ is algebraic as well. It remains to show that $\pi^*_{\rm an} s'$ is algebraic. The issue is that $\pi^*\mathcal{F}'\to \pi^*\mathcal{F}$ might not be injective, see Remark~\ref{rmk:xexpy}. Working with one $Z_i$ at a time and replacing $X$ with a suitable affine open cover of an open containing $Z_i$, we may reduce to the situation of Lemma~\ref{lemma:artin-rees} below, whose application finishes the proof.
\end{proof}

\begin{lemma} \label{lemma:artin-rees}
  Let $X$ be an affine scheme of finite type over $\CC$ and let $Z\subseteq X$ is a closed subscheme cut out by an ideal $\mathcal{I}\subseteq \cO_X$. We denote by $Z^{(n)}$ the closed subscheme of $X$ cut out by $\mathcal{I}^{n+1}$. Let $\mathcal{N}\hookrightarrow \mathcal{M}$ be an injective map between coherent $\cO_X$-modules and let $s\in \Gamma(X_{\rm an}, \mathcal{N}_{\rm an})$ be an analytic section of $\mathcal{N}$. Suppose that for all $n\geq 0$, the image of $s$ in $\Gamma(Z^{(n)}, \mathcal{M}_{\rm an}|_{Z^{(n)}})$ is algebraic. Then the image of $s$ in $\Gamma(Z^{(n)}, \mathcal{N}_{\rm an}|_{Z^{(n)}})$ is algebraic for all $n\geq 0$. 
\end{lemma}

\begin{proof}
By the Artin--Rees lemma, there exists a $k\geq 0$ such that for all $n\geq k$ we have
\[ 
  (\mathcal{I}^{n+1} \mathcal{M}) \cap \mathcal{N} \subseteq \mathcal{I}^{n-k+1}\mathcal{N}. 
\]
This yields the commutative diagram with exact rows
\[ 
  \xymatrix{
    0\ar[r] & (\mathcal{I}^{n+k+1} \mathcal{M})\cap \mathcal{N} \ar[d] \ar[r] & \mathcal{N} \ar@{=}[d] \ar[r] & {\rm im}(\mathcal{N}\to \mathcal{M}|_{Z^{(n+k)}}) \ar[d] \ar[r] & 0 \\
    0\ar[r] & \mathcal{I}^{n+1} \mathcal{N} \ar[r] & \mathcal{N} \ar[r] & \mathcal{N}|_{Z^{(n)}} \ar[r] & 0
  }
\]
which implies that the image of $s$ in $\Gamma(Z^{(n)}, \mathcal{N}_{\rm an}|_{Z^{(n)}})$ is the image of an element of 
\[
  \Gamma(Z^{(n+k)}, {\rm im}(\mathcal{N}\to \mathcal{M}|_{Z^{(n+k)}}))
\]
and hence is algebraic by Lemma~\ref{lemma:gaga-inj} applied to the inclusion ${\rm im}(\mathcal{N}\to \mathcal{M}|_{Z^{(n+k)}}) \hookrightarrow \mathcal{M}|_{Z^{(n+k)}}$.
\end{proof}

\section{Logarithmic connections and canonical extensions}
\label{s:log-conn}

\subsection{Connections on log schemes with constant log structure}
\label{ss:log-conn-loc-const}

\begin{defin} \label{def:MIC}
  Let $X$ be an idealized smooth log scheme over $\CC$ or an idealized smooth log complex analytic space. An \emph{integrable connection} on $X$ is the data of a coherent $\cO_X$-module $E$ together with an integrable logarithmic connection $\nabla\colon E\to E\otimes \Omega^1_{X/\CC}$. We denote by $\MIC(X/\CC)$ the category of integrable connections and horizontal maps. 
\end{defin}

As in the case of smooth schemes $X$ (with no log structure), the category $\MIC(X/\CC)$ is a $\CC$-linear abelian category endowed with a symmetric monoidal tensor product $E\otimes F$ and internal Hom, denoted $\underline{\rm Hom}(E, F)$. Unlike the classical case, its objects are not locally free, it is not rigid as a tensor category, and the tensor product is not exact. 

The de Rham complex $\Omega^\bullet_{X/\CC}\otimes E$ and the de Rham cohomology $H^*_{\rm dR}(X, E) = H^*(X, E\otimes\Omega^\bullet_{X/\CC})$ are defined in the usual way. We have the useful formula 
\[
  \Hom(E, F) = H^0_{\rm dR}(X, \underline{\rm Hom}(E, F)). 
\]
A morphism $f\colon X'\to X$ induces a pull-back functor 
\[
  f^*\colon \MIC(X/\CC)\to \MIC(X'/\CC)
\]
which agrees with the module pullback on underlying sheaves, and for every object $E$ of $\MIC(X/\CC)$ a map $f^* \colon H^*_{\rm dR}(X, E)\to H^*_{\rm dR}(X', f^* E)$. For an idealized smooth log scheme $X$ over $\CC$ we have an analytification functor 
\[
  E\mapsto E_{\rm an}\quad\colon\quad \MIC(X/\CC)\to \MIC(X_{\rm an}/\CC), 
\]
and a natural morphism $H^*_{\rm dR}(X, E)\to H^*_{\rm dR}(X_{\rm an}, E_{\rm an})$.

\begin{lemma} \label{lem:hollow-smooth-basics}
  Let $X$ be a hollow idealized smooth log scheme over $\CC$ or log complex analytic space. Then, the following hold.
  \begin{enumerate}[(a)]
    \item The underlying scheme $\underline{X}$ is smooth.
    \item The map $\cM^\gp_X\to \Omega^1_{X/\underline{X}}$ sending a local section $m$ to the image of $d\log(m)$ (part of the universal log derivation on $X$ relative to $\underline{X}$) annihilates the subsheaf $\cO_X^\times$ and induces an isomorphism
    \[ 
      \oM^\gp_X\otimes_\ZZ \cO_X \isomto \Omega^1_{X/\underline{X}}.
    \]
    \item The following sequence is exact.
    \[ 
      \xymatrix{0\ar[r] & \Omega^1_{\underline{X}/\CC} \ar[r] & \Omega^1_{X/\CC} \ar[r] & \oM^\gp_X\otimes_\ZZ \cO_X \ar[r] & 0}
    \]
  \end{enumerate}
\end{lemma}

\begin{proof}
Since $X$ locally admits a strict \'etale map $X\to \mathbf{A}_{P,K}$ where $K=P\setminus P^\times$, we may assume $X=\mathbf{A}_{P,P\setminus P^\times}$. Then $\underline{X}=\Spec(\CC[P^\times])$ is smooth, and the short exact sequence in question takes the form
\[ 
  \xymatrix{0\ar[r] & \CC[P^\times]\otimes P^\times \ar[r] & \CC[P^\times]\otimes P^\gp\ar[r] & \CC[P^\times]\otimes\ov{P}^\gp\ar[r] & 0.} \qedhere
\]
\end{proof}

In this subsection, we shall explicate the category $\MIC(X/\CC)$ in the case when $X$ is hollow and idealized smooth. For the sake of intuition, we note that the Betti realization of such an $X$ is a torus bundle over $\underline{X}_{\rm an}$. More precisely, the map $\tau \colon X_{\rm log} \to \underline{X}_{\rm an}$ is a torsor under the family of real tori $\underline{\Hom}(\oM_X, \mathbf{S}^1)$. A~splitting $\eps$ of $\cM_X$ provides a section of this torsor (see Figure~\ref{fig:bundle}), and a pull-back functor $\epsast$ from local systems on $X_{\rm log}$ to local systems on $\underline{X}_{\rm an}$. We use the superscript $\circledast$ here to indicate that the functor is not induced by a map of log schemes $\underline{X}\to X$, though it behaves as if it was. The K\"unneth theorem allows one to identify local systems on $X_{\rm log}$ with local systems on $\underline{X}$ endowed with an action of $\underline{\Hom}(\oM_X, \ZZ(1))$. 

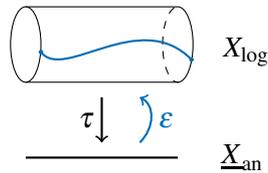
\begin{figure}[h!] \label{fig:bundle}
  \begin{tikzpicture}
    \begin{scope}[shift={(0,0)},xscale=-1]
      \begin{scope}
        \clip (-1, -1) rectangle (-2, 1);
      \draw (-1,0) ellipse (0.2 and 0.5);
      \end{scope}

      \begin{scope}
        \clip (-1, -1) rectangle (1, 1);
        \draw[dashed] (-1,0) ellipse (0.2 and 0.5);
      \end{scope}

      \draw (1,0) ellipse (0.2 and 0.5);
      \draw (-1, 0.5) -- (1, 0.5);
      \draw (-1, -0.5) -- (1, -0.5);

      \draw[NavyBlue,thick] (-1.18, -0.2) .. controls (-0.5, 0.5) and (0.5, -0.5) .. (0.8, -0.1); 

      \fill[NavyBlue] (-1.18, -0.2) circle (0.03);
      \fill[NavyBlue] (0.8, -0.1) circle (0.03);
    \end{scope}

    \begin{scope}[shift={(0, -1.5)}]
      \draw[thick] (-1, 0) -- (1, 0);
    \end{scope}

    \draw[->,thick] (0, -0.7) -- (0, -1.3);
    \draw[->,NavyBlue,thick] (0.5, -1.3) .. controls (0.7, -1.1) and (0.7,-0.8) .. (0.5, -0.7);

    \draw[NavyBlue] (.85, -1) node {$\eps$};
    \draw (-.2, -1) node {$\tau$};
    \draw (1.9, -0.1) node {$X_{\rm log}$};
    \draw (1.85, -1.5) node {$\underline{X}_{\rm an}$};
  \end{tikzpicture}
  \caption{The Kato--Nakayama space of a log scheme $X$ with a splitting $\eps$.}
\end{figure}

Proceeding by analogy, we shall now describe integrable connections on $X$ in terms of integrable connections on $\underline{X}$ depending on the choice of a splitting $\eps$. 

\begin{constr} \label{constr:epsast}
Let $X$ be a hollow idealized smooth log scheme over $k$ and let $\eps$ be a splitting on $X$. Although $\eps$ (or rather the corresponding retraction $\eps'\colon\cM_X\to\cO^\times_X = \cM_{\underline{X}}$) does not come from a morphism of log schemes $\underline{X}\to X$, it induces a morphism 
\[
  \epsast\colon \Omega^1_{X/\CC}\ra \Omega^1_{\underline{X}/\CC}
\]
as follows. Recall that giving a map $\Omega^1_{X/\CC}\to \mathcal{F}$ of $\cO_X$-modules is equivalent to giving a log derivation with values in $\mathcal{F}$, which by definition is a pair of maps $(\partial, \partial^{\rm log})$ where $\partial\colon \cO_X\to \mathcal{F}$ is a derivation and where $\partial^{\rm log}\colon \cM_X\to\mathcal{F}$ is a monoid homomorphism satisfying $\alpha(m)\partial^{\rm log}(m) = \partial(\alpha(m))$ for every local section $m$ of $\cM_X$. Let $d^{\rm log}_\eps\colon \cM_X\to \Omega^1_{\underline{X}/\CC}$ be the map sending a local section $m$ to $d\log(\eps'(m))$ and let $d\colon \cO_X\to \Omega^1_{\underline{X}/\CC}$ be the universal derivation on $\underline{X}$. We claim that the pair $(d, d_\eps^{\rm log})$ forms a log derivation on $X/\CC$. By definition, we need to verify the formula
\[
  \alpha(m)d_\eps(m) = d\alpha(m)
\]
for a local section $m$ of $\cM_X$. By Lemma~\ref{lem:splitting-basics}, $\alpha(m)$ is nilpotent if $m\notin \cO_X^\times$, so $\alpha(m)=0$ as $\underline{X}$ is reduced, and the formula is satisfied in this case. For $m\in\cO_X^\times$, the formula is obvious. We let $\epsast$ be the $\cO_X$-linear map corresponding to this log derivation by the universal property of $\Omega^1_{X/\CC}$. 
\end{constr}

Let $\mathcal{Q}_X = \Omega^1_{X/\underline{X}} = \oM_X^\gp\otimes_\ZZ \cO_X$ (Lemma~\ref{lem:hollow-smooth-basics}). The map $\epsast$ splits the short exact sequence introduced in Lemma~\ref{lem:hollow-smooth-basics}
\[
  \xymatrix{0\ar[r] & \Omega^1_{\underline{X}/\CC} \ar[r] & \Omega^1_{X/\CC} \ar[r]^q & \mathcal{Q}_X  \ar[r] & 0}
\]
Using the induced decomposition $\Omega^1_{X/\CC} \simeq \mathcal{Q}_X\oplus \Omega^1_{\underline{X}/\CC}$, we see that the data of a logarithmic connection $E\to E\otimes\Omega^1_{X/\CC}$ amounts to the data of a pair of maps $(\epsast(\nabla), \rho_\nabla)$ defined as
\[ 
  \epsast(\nabla) \colon E \xlongrightarrow{\nabla} E\otimes \Omega^1_{X/\CC} \xto{1\otimes \epsast} E\otimes \Omega^1_{\underline{X}/\CC},
  \qquad
  \rho_\nabla \colon E\xlongrightarrow{\nabla} E\otimes \Omega^1_{X/\CC} \xto{1\otimes q} E\otimes \mathcal{Q}_X 
\]
where $\epsast(\nabla)$ is a connection on $\underline{X}$ and where $\rho_\nabla$ is $\cO_X$-linear. 

\begin{lemma} \label{lem:int-Higgs-field}
  Let $\nabla\colon E\to E\otimes \Omega^1_{X/\CC}$ be a logarithmic connection on an $\cO_X$-module $E$. Then $\nabla$ is integrable if and only if the following conditions hold:
  \begin{enumerate}[i.]
    \item the connection $\epsast(\nabla)\colon E\to E\otimes\Omega^1_{\underline{X}/\CC}$ is integrable,
    \item the map $\rho_\nabla\colon E\to E\otimes \mathcal{Q}_X$ is a $\mathcal{Q}_X$-Higgs field,
    \item the map $\rho_\nabla\colon E\to E\otimes \mathcal{Q}_X$ is horizontal, where we endow $\mathcal{Q}_X = \oM^\gp_X\otimes_\ZZ \cO_X$ with the canonical connection satisfying $\nabla(m\otimes 1) = 0$, and $E\otimes\mathcal{Q}_X$ with the tensor product connection.
  \end{enumerate}
\end{lemma}

Before we give the (straightforward) proof, let us recall the notion of a Higgs field used in its statement. We will make use of the extra generality shortly.

\begin{defin} \label{def:Higgs-object}
  Let $(\mathcal{A}, \otimes)$ be a symmetric monoidal abelian category.
  \begin{enumerate}[(a)]
    \item For an object $F$ of $\mathcal{A}$ and $r\geq 0$, the \emph{$r$-th exterior power} $\bigwedge^r F$ is the image of the antisymmetrization map $\sum_{\sigma\in S_r} (-1)^{{\rm sign}(\sigma)}\sigma \colon F^{\otimes r}\to F^{\otimes r}$, see \cite[\S 7, p.\ 165]{Deligne_CategoriesTannakiennes}. For $i,j\geq 0$, we denote by $\wedge$ the map
    \[ 
      \wedge\colon \bigwedge^i F \otimes \bigwedge^j F \ra \bigwedge^{i+j} F.
    \]
    \item Let $F$ be an object of $\mathcal{A}$. An \emph{$F$-Higgs object} of $\mathcal{A}$ is a pair $(E, \theta)$ where $E$ is an object of $\mathcal{A}$ and $\theta\colon E\to E\otimes F$ is a morphism (called a \emph{Higgs field}) satisfying $\theta^1\circ\theta = 0$ where $\theta^1$ is the map
    \[ 
      E\otimes F\xto{-\theta\otimes 1} E\otimes F\otimes F \xto{1\otimes \wedge} E\otimes \bigwedge^2 F,
    \]
    A morphism of $F$-Higgs objects $(E',\theta')\to (E,\theta)$ is a morphism $f\colon E'\to E$ for which the square
    \[ 
      \xymatrix{
        E' \ar[r]^-{\theta'} \ar[d]_f &  E'\otimes F \ar[d]^{f\otimes 1} \\
        E\ar[r]_-\theta & E\otimes F
      }
    \]
    commutes. We denote by ${\rm HIG}_F(\mathcal{A})$ the category of $F$-Higgs objects of $\mathcal{A}$.\footnote{If $F$ is dualizable, then the data of an $F$-Higgs field on an object $E$ is equivalent to the structure of an $\op{Sym}(F^\vee)$-module on $E$. As pointed out by Langer \cite[\S 4, especially Remark~4.12]{LangerHiggsNormal}, in general it is better to work with the category of $\op{Sym}(L)$-modules for an object $L$ of $\mathcal{A}$ --- for example, unlike the category of $F$-Higgs objects in our sense, it is always an abelian category. Fortunately in our setting the object $F$ will always be dualizable.}
    \item Let $(E, \theta)$ be an $F$-Higgs object of $\mathcal{A}$. The \emph{Higgs complex} of $(E, \theta)$ is the complex
    \[ 
      E \ra E\otimes F \ra E\otimes \bigwedge^2 F \ra \cdots
    \]
    where the differential $\theta^i\colon E\otimes \bigwedge^i F \to E\otimes \bigwedge^{i+1} F$ is the composition
    \[ 
      E\otimes \bigwedge^i F \xto{(-1)^i\theta\otimes 1} E\otimes F\otimes \bigwedge^i F \xto{1\otimes \wedge} E\otimes \bigwedge^{i+1} F. 
    \]
  \end{enumerate}
\end{defin}

\begin{proof}[Proof of Lemma~\ref{lem:int-Higgs-field}]
Using the decomposition $\Omega^2_{X/\CC} \simeq \Omega^2_{\underline{X}/\CC} \oplus (\mathcal{Q}_X\otimes \Omega^1_{\underline{X}/\CC}) \oplus \bigwedge^2 \mathcal{Q}_X$ induced by $\epsast$, we can describe the sequence $E\to E\otimes \Omega^1_{X/C}\to E\otimes\Omega^2_{X/\CC}$ as the totalization of the diagram
\[ 
  \xymatrix{
    E\otimes \Omega^2_{\underline{X}/\CC} \\
    E\otimes \Omega^1_{\underline{X}/\CC} \ar[u]^{\epsast(\nabla)^1} \ar[r]^-{\rho_\nabla\otimes 1} & E\otimes \mathcal{Q}_X\otimes \Omega^1_{\underline{X}/\CC} \\
    E \ar[r]_{\rho_\nabla} \ar[u]^{\epsast(\nabla)} & E\otimes \mathcal{Q}_X\ar[u]_{\delta} \ar[r]_{\rho_\nabla^1} & E\otimes \bigwedge^2\mathcal{Q}_X.
  }
\]
where $\delta$ is the connection on $E\otimes \mathcal{Q}_X$, explicitly described as $\delta(e\otimes q) = s(\nabla(e)\otimes q)$ for $e\in E$ and $q\in \oM_X^\gp$ where $s\colon E\otimes\Omega^1_{\underline{X}/\CC}\otimes\mathcal{Q}_X\isomto E\otimes\mathcal{Q}_X\otimes\Omega^1_{\underline{X}/\CC}$ is the shuffle map. Therefore $\nabla^1\circ \nabla = 0$ if and only if $\epsast(\nabla)^1 \circ \epsast(\nabla) = 0$, $\rho_\nabla^1\circ\rho_\nabla =0 $, and the square in the diagram commutes. 
\end{proof}

\begin{remarks} \label{rmks:epsast}
\begin{enumerate}[1.] 
  \item \label{rmk:Higgs-residue}
    Suppose that $\oM_X$ is constant with value $\NN^r$. The above shows that in the presence of a splitting $\eps$ on $X$, the data of a log integrable connection on $X$ is equivalent to the data of an integrable connection on $\underline{X}$ endowed with $r$ commuting horizontal endomorphisms (the ``residues''). 
  \item \label{rmk:functoriality}
    Let $f\colon Y\to X$ be a map between hollow idealized smooth log schemes over $\CC$ and let $\eps_X$ and $\eps_Y$ be splittings with which $f$ is compatible (Remark~\ref{rmk:Xsplit-functorial}). Then, we have an isomorphism of functors $\epsast_Y(f^*(-)) \simeq f^*(\epsast_X(-))$ from $\MIC(X/\CC)$ to $\MIC(\underline{Y}/\CC)$.
  \item \label{rmk:}
    The following variant is quite useful in practice. Let $(\underline{X}, D)$ be an snc pair and let $X$ be the associated log scheme. Let $D_1, \ldots, D_r$ be the components of $D$. For $I\subseteq \{1, \ldots, r\}$, let $Z=\bigcap_{i\in I} D_i$ be the closure of a stratum of $D$ and let $\partial Z = Z\cap \bigcup_{i\notin I} D_i$. Then $(Z, \partial Z)$ is an snc pair, and we denote by $Z'$ the corresponding log scheme. We have a morphism of log schemes $Z\to Z'$ which is the identity on the underlying schemes. This map \'etale locally admits a ``splitting,'' i.e.\ a splitting $\eps$ of the map $\cM_{Z'}\to \cM_Z$. In turn, the splitting $\eps$ induces a functor $\epsast\colon \MIC(Z/\CC)\to\MIC(Z'/\CC)$ enjoying properties similar to those listed in Proposition~\ref{prop:Higgs-description} below. See \cite[\S 3.3]{LangerNearby} for related calculations.
  \item \label{rmk:algebroids}
    The dual map $(\epsast)^\vee \colon T_{\underline{X}/\CC}\to T_{X/\CC}$ defines a morphism of Lie algebroids, i.e.\ it is a morphism of $\CC$-Lie algebras on $\underline{X}$ (see \cite[\S 2]{LangerLie}) such that $u\circ (\epsast)^\vee$ is the identity where $u\colon T_{X/\CC}\to T_{\underline{X}/\CC}$ is the map induced by the map of log schemes $X\to \underline{X}$. In particular, we obtain a functor on the level of modules over those algebroids which agrees with $\epsast$. 
  \item \label{rmk:D-mod}
    The functor $\epsast\colon \MIC(X/\CC)\to \MIC(\underline{X}/\CC)$ extends to a functor
    \[ 
      \epsast \colon \cat{Mod}_{\rm coh}(\mathcal{D}_X) \ra \cat{Mod}_{\rm coh}(\mathcal{D}_{\underline{X}})
    \]
    between the categories of coherent (log) $D$-modules \cite{KoppensteinerTalpo,Koppensteiner}.
\end{enumerate}
\end{remarks}

Summarizing the above discussion, we have proved:

\begin{prop} \label{prop:Higgs-description}
  Let $X$ be a hollow idealized smooth log scheme over $\CC$ and let $\eps$ be a splitting on $X$. Then, the following assertions hold.
  \begin{enumerate}[(a)]
    \item The sheaf $\oM_X$ is locally constant, $\underline{X}$ is smooth over $\CC$, and $\eps$ induces a map 
    \[
      \epsast\colon\Omega^1_{X/\CC}\ra \Omega^1_{\underline{X}/\CC}
    \]
    splitting the exact sequence
    \[ 
      \xymatrix{0\ar[r] & \Omega^1_{\underline{X}/\CC}\ar[r] & \Omega^1_{X/\CC} \ar[r] & \mathcal{Q}_X\ar[r] & 0,} \quad \mathcal{Q}_X = \ov{\cM}^{\rm gp}_X\otimes_{\ZZ} \cO_X .
    \]
    \item In turn, the map $\epsast$ induces a monoidal functor
    \[ 
      \epsast \colon \MIC(X/\CC) \ra \MIC(\underline{X}/\CC)
    \]
    and an equivalence of categories
    \[ 
      (E, \nabla)\mapsto (\epsast(E, \nabla), \rho_\nabla) \quad\colon \quad \MIC(X/\CC) \isomto {\rm HIG}_{\mathcal{Q}_X}(\MIC(\underline{X}/\CC)).
    \]
    \item The equivalence in (b) is compatible with cohomology in the following way. For an object $(E,\nabla)$ of $\MIC(X/\CC)$, we have a functorial isomorphism between the de Rham complex $E\otimes \Omega^\bullet_{X/\CC}$ of $E$ and the total complex of the double complex
    \[ 
      \xymatrix{
        & \cdots & \cdots \\
        \cdots \ar[r] & E \otimes \bigwedge^i \mathcal{Q}_X \otimes \Omega^{j+1}_{\underline{X}/\CC} \ar[r] \ar[u] & E \otimes \bigwedge^{i+1} \mathcal{Q}_X \otimes \Omega^{j+1}_{\underline{X}/\CC} \ar[r] \ar[u] & \cdots \\
        \cdots \ar[r] & E \otimes \bigwedge^i \mathcal{Q}_X \otimes \Omega^{j}_{\underline{X}/\CC} \ar[r] \ar[u] & E \otimes \bigwedge^{i+1} \mathcal{Q}_X \otimes \Omega^j_{\underline{X}/\CC} \ar[r] \ar[u] & \cdots \\
        & \cdots \ar[u] & \cdots \ar[u]         
      }
    \]
    whose $i$-th column is the de Rham complex of $E\otimes \bigwedge^i \mathcal{Q}_X$ and whose $j$-th row is the Higgs complex of $E$ tensored with $\Omega^j_{\underline{X}/\CC}$.
  \end{enumerate}
\end{prop}

\begin{cor} \label{cor:hollow-tannakian}
  Let $X$ be a hollow idealized smooth log scheme over $\CC$. Then, every object of $\MIC(X/\CC)$ is locally free as an $\cO_X$-module, and the category $\MIC(X/\CC)$ is a rigid tensor category \cite[\S 2]{Deligne_CategoriesTannakiennes}. If $X$ is connected, then for every point $x\in X(\CC)$ and every splitting $\eps$ on $x$, the functor 
  \[
    \omega=\epsast\circ x^*\colon \MIC(X/\CC)\ra \MIC(\underline{x}/\CC) = \cat{Vect}_\CC
  \]
  is an exact faithful $\CC$-linear tensor functor, making $(\MIC(X/\CC), \omega)$ a Tannakian category.
\end{cor}

\begin{proof}
Since $X$ locally admits a splitting, for every object $(E, \nabla)$ of $\MIC(X/\CC)$, the $\cO_X$-module $E$ locally admits the structure of an object of $\MIC(\underline{X}/\CC)$ and hence is locally free. The rest follows.
\end{proof}

The following straightforward lemma records the dependence of $\epsast(E)$ on the splitting $\eps$.

\begin{lemma} \label{lem:dependence-on-eps}
  Let $X$ be a hollow idealized smooth log scheme over $\CC$ and let $\eps_i$ ($i=0,1$) be two splittings on $X$. We treat $\eps_0/\eps_1$ as a map $\oM_X\to \cO_X^\times$. The composition $d\log(\eps_0/\eps_1)\colon \oM_X\to\cO_X^\times\to \Omega^1_{\underline{X}/\CC}$ is additive and hence uniquely extends to an $\cO_X$-linear map $\delta(\eps_0, \eps_1)\colon \mathcal{Q}_X\to \Omega^1_{\underline{X}/\CC}$. Let $(E, \nabla)$ be an object of $\MIC(X/\CC)$. Then, the $\cO_X$-linear map $\epsast_0(\nabla) - \epsast_1(\nabla)$ equals the composition
  \[ 
    \xymatrix@C=1.5cm{
      E \ar[r]^-{\rho_\nabla} & E\otimes \mathcal{Q}_X \ar[r]^-{1\otimes \delta(\eps_0,\eps_1)} & E\otimes \Omega^1_{\underline{X}/\CC}.
    }
  \]
\end{lemma}

Finally, we state a result for a general idealized smooth (non-hollow) log scheme $X$. Intuitively, it says that the ``map'' $\eps_{\rm univ}\circ \pi\colon \underline{X}^\#\to X$ is ``unramified.''

\begin{lemma} \label{lem:xsharp-x-unram}
  Let $X$ be an idealized smooth log scheme over $\CC$ and let $\pi\colon X^\#\to X$ be the morphism constructed in \S\ref{ss:splittings}, endowed with the universal splitting $\eps_{\rm univ}$. Then, the composition
  \[
    \xymatrix{\pi^* \Omega^1_{X/\CC} \ar[r]^{\pi^*} & \Omega^1_{X^\#/\CC} \ar[r]^-{\epsast_{\rm univ}} & \Omega^1_{\underline{X}^\#/\CC}}
  \]
  is an isomorphism.
\end{lemma} 

\begin{proof}
Since the log stratification $p\colon X^\flat\to X$ satisfies $p^*\Omega^1_{X/\CC}\simeq \Omega^1_{X^\flat/\CC}$, we may assume that $X$ is hollow. Working locally, by Lemma~\ref{lem:splitting-basics}(d) we may assume that $X=\Spec(P\to \CC)\times \mathbf{G}_m^n$ for some sharp fs monoid $P$, where $P\to \CC$ maps $P\setminus 0$ to $0$. In this case, $X^\# = \Spec(P\to \CC[P^\gp])\times \mathbf{G}_m^n$ where $P\to \CC[P^\gp]$ maps $P\setminus 0$ to $0$, and $\varepsilon_{\rm univ}(p) = (p,p)$ (see Example~\ref{ex:AP-univ-splitting}). The rest follows by a direct calculation using the formula
\[
  \epsast_{\rm univ}\pi^*(d\log(p)) = d\log(p). \qedhere
\]
\end{proof}

\subsection{Ogus' Riemann--Hilbert correspondence}

Ogus' correspondence \cite{OgusRH} describes the category $\MIC(X/\CC)$ of integrable connections on an idealized smooth log complex analytic space $X$ in terms of certain constructible sheaves on the Kato--Nakayama space $X_{\rm log}$. 

At its kernel, it is a generalization of the following basic fact. A locally free log connection on the open unit disc in $\mathbf{A}_\NN^{\rm an}$ can be put in the form $(\cO_X^n, d+Ud\log z)$ where $U\in M_{n\times n}(\CC)$ is a constant matrix. The monodromy of the associated local system on the punctured disc is given by $T = \exp(-2\pi i U)$ \cite[Chap.\ II, Lemme~1.17.1]{DeligneLNM163}. In order to recover the log connection from its monodromy, we need to describe the set of ``logarithms'' of the matrix $T$, which is in bijection with the set of $\CC$-gradings $\CC^n=\bigoplus_{\lambda\in\CC} V_\lambda$ stable under $U$ and such that $U$ has unique eigenvalue $\exp(-2\pi i\lambda)$ on $V_\lambda$ for every $\lambda\in\CC$, determined by the condition that $T$ has unique eigenvalue $\lambda$ on $V_\lambda$.

\begin{defin}[{\cite[Definition~3.2.4]{OgusRH}}] \label{def:ogus-Lcoh}
  Let $X$ be an idealized smooth log complex analytic space. We set $\ov{K}_X = \alpha^{-1}(0)/\cO_X^\times \subseteq \oM_X$, and define
  \[
    \CC_X^{\rm log} = \CC[-\tau^* \oM_X]/(-\ov{K}_X), \qquad
    \bLambda_X = \CC \otimes \tau^* \oM_X^\gp
  \]
  where $\tau\colon X_{\rm log}\to X$ is the Kato--Nakayama space of $X$. The inclusion $-\tau^*\oM_X\hookrightarrow \bLambda$ makes $\CC_X^{\rm log}$ into a $\bLambda_X$-graded sheaf of rings on $X_{\rm log}$. We also set $I_x = \Hom(\oM_{X,x}^\gp, \ZZ(1)) = \pi_1(\tau^{-1}(x))$ for $x\in X$.

  We denote by $L(X)$ the category (denoted $L^\bLambda_{\rm coh}(\CC_X^{\rm log})$ in \cite{OgusRH}) of $\bLambda_X$-graded $\CC_X^{\rm log}$-modules $V$ satisfying the following conditions:
  \begin{enumerate}
    \item $V$ is log constructible \cite[Definition~3.2.3]{OgusRH},
    \item for every $x\in X_{\rm log}$, the stalk $V_x$ is finitely generated over $\CC_{X,x}^{\rm log} = \CC[-\oM_{X,\tau(x)}]/(-\ov{K}_{X,\tau(x)})$,
    \item for every log path \cite[p.\ 704]{OgusRH} from $x$ to $y$ in $X_{\rm log}$, the cospecialization map 
    \[ 
      \gamma_{x,y}^* \colon V_x \otimes_{\CC_{X,x}^{\rm log}} \CC_{X,y}^{\rm log} \ra V_y
    \] 
    is an isomorphism,
    \item for every $x\in X_{\rm log}$, every $\gamma\in I_{\tau(x)}$ and every $\lambda\in \bLambda_{X,x}$, the number $\exp \langle \lambda, \gamma \rangle$ is the unique eigenvalue of $\gamma\colon V_{x,\lambda}\to V_{x,\lambda}$.
  \end{enumerate} 
\end{defin} 

\begin{thm}[{\cite[Theorem 3.4.2]{OgusRH}}] \label{thm:ogus-RH}
  Let $X$ be an idealized smooth log complex analytic space.\footnote{In \cite{OgusRH} this theorem is stated for analytifications of idealized smooth log schemes over $\CC$, but the proof does not make use of that assumption.} There is an equivalence of tensor categories
  \[ 
    \mathcal{V}\colon \MIC(X/\CC) \isomto L(X),
  \]
  functorial with respect to maps $Y\to X$ of idealized smooth log complex analytic spaces. Furthermore, there is a functorial (in both $E$ and $X$) quasi-isomorphism
  \[ 
    E\otimes\Omega^\bullet_{X/\CC} \ra R\tau_*(\mathcal{V}(E))_0
  \]
  where the subscript $0$ denotes the zeroth graded piece.
\end{thm}

\begin{remark}[Functoriality of $\mathcal{V}$]
Functoriality of the equivalence $\mathcal{V}$ with respect to $X$ means the following. Let $f\colon Y\to X$ be a map of idealized smooth log complex analytic spaces. The pullback functor
\[ 
  f^*\colon L(X)\ra L(Y)
\]
is induced by the module pullback functor with respect to the map of ringed spaces 
\[
  (Y_{\rm log}, \CC_Y^{\rm log})\to (X_{\rm log}, \CC_X^{\rm log}),
\]
with $\Lambda_Y$-grading induced by the $\Lambda_X$ grading via the map $f^* \Lambda_X\to \Lambda_Y$. Then, functoriality means the existence of a natural isomorphism between the two compositions in the square below.
\begin{equation} \label{eqn:ogus-RH-functoriality} 
  \xymatrix@C=2.5cm{
    \MIC(X/\CC) \ar[r]_-{\sim}^-{\mathcal{V}_Y} \ar[d]_{f^*} & L(X) \ar[d]^{f^*} \\
    \MIC(Y/\CC) \ar[r]^-\sim_-{\mathcal{V}_X} & L(Y).  
  }
\end{equation}
\end{remark}

\begin{example} \label{ex:L-of-AP}
  Let $P$ be a sharp fs monoid and let $K\subseteq P$ be an ideal. Let $L(P,K)$ be the category of finitely generated $P^\gp\otimes\CC$-graded $\CC[P]/(K)$-modules $V = \bigoplus_{\lambda\in P^\gp\otimes\CC}V_\lambda$ endowed with a homogeneous action of $\pi_1(P) = \Hom(P^\gp, \ZZ(1))$ such that for every $\gamma\in \pi_1(P)$ and every $\lambda\in P^\gp\otimes\CC$, the number $\exp\langle\lambda,\gamma\rangle$ is the unique eigenvalue of $\gamma\colon V_\lambda\to V_\lambda$. Let $X = \mathbf{A}_{P,K}^{\rm an}$. Then by \cite[Proposition~3.2.5]{OgusRH} there is a tensor equivalence
  \[ 
    L(X_{\rm an}) \simeq L(-P, -K).
  \]
\end{example}

\begin{lemma} \label{lem:local-RH}
  Let $P$ be a sharp fs monoid and let $K\subseteq P$ be an ideal. Let $F\subseteq P$ be a face disjoint from $K$ and let $P' = P_F$, $K' = K + P'$. Let $X = \mathbf{A}_{P,K}^{\rm an} \subseteq \Hom(P, \CC)$ and $X' = \mathbf{A}_{P',K'}^{\rm an} \subseteq X$. For a generating set $S$ of $P$ and a real number $\delta>0$, consider the following open subsets:
  \begin{align*} 
    U_\delta &= \{ x\in X \,:\, \text{$|x(p)|<\delta$ for $p\in S$} \}, \\
    U'_\delta &= U_\delta\cap X' 
  \end{align*}
  Then, there are tensor equivalences $L(U_\delta) \simeq L(-P, -K)$ and $L(U'_\delta)\simeq L(-P', K')$ fitting inside a commutative square
  \[ 
    \xymatrix{
      L(U_\delta) \ar[r] \ar[d]^{\displaystyle\wr} & L(U'_\delta)  \ar[d]^{\displaystyle\wr} \\
      L(-P, -K) \ar[r] & L(-P', -K').
    }
  \]
\end{lemma}

\begin{proof}
We first show that for every face $F$ of $P$ disjoint from $K$, the inclusion of $U_\delta\cap X_F\hookrightarrow X_F$ is a homotopy equivalence. Here $X_F = \mathbf{A}_{F^\gp}$ is the stratum corresponding to $F$. To show this we may ignore the ideal $K$. We also note that $S\cap F$ is a generating set for $F$. Identifying $X$ with $\Hom(P, \CC)$, we have 
\begin{align*}
  X_F &= \{x\colon P\to \CC\,:\, x^{-1}(\CC^\times) = F \} \\
    &= \Hom(F^\gp, \CC^\times) 
    = \Hom(F^\gp, \mathbf{R}_{>0}) \times \Hom(F^\gp, \mathbf{S}^1),\\[.8em]
  X_F\cap U_\delta &= \{x\colon F^\gp\to \CC^\times \,:\, \text{$|x(p)|<\delta$ for $p\in S\cap F$} \} \\
    &= \{x\colon F^\gp\to \mathbf{R}_{>0}  \,:\, \text{$x(p)<\delta$ for $p\in S\cap F$} \}
      \times \Hom(F^\gp, \mathbf{S}^1).
\end{align*}
It remains to observe that the set cut out by the equations $x(p)<\delta$ in the convex cone $\Hom(F^\gp, \mathbf{R}_{>0})$ is convex and contains a neighborhood of zero, and hence is contractible, so that its inclusion into $\Hom(F^\gp, \mathbf{R}_{>0})$ is a homotopy equivalence.

The following description of log constructible sheaves on $X^{\rm log}$ where $X=\mathbf{A}_{P,K}$ appears in \cite[\S 3.2]{OgusRH}. Such a sheaf corresponds to the data of a functor
\[ 
  \mathcal{F}\colon \{\text{faces of $P$}\} \ra \pi_1(P)\text{-sets}
\]
where $\mathcal{F}(F)$ is the stalk of the corresponding sheaf at a prescribed point of $X_F^{\rm log}$, and where the maps $\mathcal{F}(F')\to \mathcal{F}(F)$ for $F'\subseteq F$ correspond to the cospecialization maps. In these terms, the sheaf $\CC_X^{\rm log}$ corresponds to the family of algebras $\CC[-P/F]/(-K)$, the sheaf of gradings $\bLambda_X$ corresponds to the family of groups $(P/F)^\gp\otimes\CC$, and an object $V$ of $L(X)$ corresponds to the family of finitely generated $(P/F)^\gp\otimes\CC$-graded $\CC[-P/F]/(-K)$-modules $V_F$ endowed with an action of $\pi_1(P)$ satisfying the suitable condition on eigenvalues, and such that the cospecialization map for the inclusion $0\subseteq F$ induces an isomorphism $V_F = V_0\otimes_{\CC[-P]/(-K)}\CC[-P/F]/(-K)$. It is therefore determined by the object $V_0\in L(-P,-K)$ (see the discussion following \cite[Definition~3.2.4]{OgusRH}), and so $L(X)$ is equivalent to the category $L(-P,-K)$. We have just explained the proof of \cite[Proposition~3.2.5]{OgusRH} mentioned in Example~\ref{ex:L-of-AP} above.

The point of going into Ogus' proof is that his argument for this description, employing the universal cover of $X^{\rm log}$, applies equally well to any open subset $U\subseteq X^{\rm an}$ with the property that for every face $F$ of $P$ disjoint from $F$, the preimage of $U\cap X_F$ in the universal cover of $X^{\rm log}$ is contractible.  In particular, by the result of the first paragraph, it applies to the sets $U_\delta$ and $U'_\delta$, and we obtain our assertion.
\end{proof}

\subsection{Canonical extensions}
\label{ss:can-ext}

In this section, we fix a good embedding $j\colon X\hookrightarrow \ov{X}$ of log complex analytic spaces (Definition~\ref{def:good-embedding}). For $x\in \ov{X}$, we have a canonical isomorphism
\[ 
  \oM_{\ov{X},x} \simeq \mathcal{F}_x \oplus (j_*\oM_{X})_x
\]
where $\mathcal{F}_x\simeq \NN^s$ for some $s\geq 0$ depending on $x$. For an object $E$ of $\MIC(\ov{X}/\CC)$, the generators $e_1, \ldots, e_s$ (well  defined up to permutation) of $\mathcal{F}_x$ induce linear endomorphisms (``residues'')
\[ 
  \rho_i \colon E_x \ra E_x, \qquad E_x = E\otimes \kappa(x).
\]
We denote by $R_x\subseteq\CC$ the union of the set of negatives of the eigenvalues of $\rho_1, \ldots, \rho_s$ and call it the set of \emph{exponents at infinity} of $E$ at $x$. 

\begin{defin} \label{def:tau-adapted}
  Let $\tau\colon \CC/\ZZ\to\CC$ be a section of the projection $\CC\to\CC/\ZZ$ and let $j\colon X\to \ov{X}$ be a good embedding. We say that an object $E\in \MIC(\ov{X}/\CC)$ is \emph{$\tau$-adapted} if:
  \begin{enumerate}
    \item for every $x\in \ov{X}$, the set $R_x$ of exponents at infinity of $E$ at $x$ is contained in ${\rm im}(\tau)$,
    \item the map $E\to j_* j^* E$ is injective.
  \end{enumerate} 
  We denote by $\MIC^\tau(\ov X, X/\CC)$ the full subcategory of $\MIC(\ov X/\CC)$ consisting of $\tau$-adapted objects.
\end{defin}

\begin{thm} \label{thm:canonical-ext} 
  Let $\tau\colon \CC/\ZZ\to\CC$ be a section of the projection $\CC\to\CC/\ZZ$. 
  \begin{enumerate}[(a)]
    \item The restriction functor induces an equivalence
    \[
      \MIC^\tau(\ov X, X/\CC)\isomto \MIC(X/\CC).
    \]
    \item If $\tau(0)=0$, then for every object $E$ of $\MIC^\tau(\ov X, X/\CC)$ the map
    \[
      j^* \colon H^*_{\rm dR}(\ov{X}, E)\ra H^*_{\rm dR}(X, j^* E)
    \]
    is an isomorphism.
    \item An object $E$ of $\MIC^\tau(\ov X, X/\CC)$ is locally free (as an $\cO_X$-module) if and only if its restriction to $X$ is locally free.
  \end{enumerate}
\end{thm} 

Before proving the result, we show an auxiliary statement about monoids.

\begin{lemma} \label{lem:canonical-equivariant}
  Let $P$ be a sharp fs monoid and let $Q = P\times \NN^r$ and $Q' = P\times \ZZ^r$. Let $K\subseteq P$ be an ideal, and denote also by $K$ the ideal it generates in $Q$ and $Q'$. Denote by $L^\tau(Q, K)$ the full subcategory of $L(Q, K)$ (see Example~\ref{ex:L-of-AP}) consisting of objects $V$ such that
  \begin{enumerate}
    \item $V$ is generated as a $\CC[Q]$-module by $\bigoplus_{\lambda\in (P^\gp\otimes \CC)\times {\rm im}(\tau)^r} V_\lambda$,
    \item the map $V\to V' = V\otimes_{\CC[Q]} \CC[Q']$ is injective.
  \end{enumerate}
  Then
  \begin{enumerate}[(a)]
    \item the functor $V\mapsto V' = V\otimes_{\CC[Q]} \CC[Q']$ induces an equivalence of categories
    \[ 
      L^\tau(Q, K)\isomto L(Q', K),
    \] 
    \item an object $V$ of $L^\tau(Q)$ is free as a $\CC[Q]/(K)$-module if and only if $V'$ is free as a $\CC[Q']/(K)$-module.
  \end{enumerate}
\end{lemma}

\begin{proof}
Let $\Sigma = \{z\in \CC\,:\, z-\tau(z) \in \NN\}$. For an object $V'$ of $L(Q', K)$, we set
\[ 
  V = \bigoplus_{\lambda \in (P^\gp\otimes\CC)\times \Sigma^r} V'_\lambda.
\]
This is a $Q^\gp\otimes\CC$-graded $\CC[Q]$-module with an action of $\pi_1(Q) = \pi_1(Q')$ satisfying the required condition on eigenvalues. We check that it is finitely generated by homogeneous elements lying in degrees in 
\[
  (P^\gp\otimes\CC)\times {\rm im}(\tau)^r.
\]
Let $v'_1, \ldots, v'_r$ be homogeneous generators of $V'$ lying in degrees $\lambda'_i = (\eta_i, \lambda'_{i,1}, \ldots, \lambda'_{i,r})$. Let 
\[ 
  v_i = x_1^{\tau(\lambda'_{i,1})-\lambda'_{i,1}} \cdots x_r^{\tau(\lambda'_{i,r})-\lambda'_{i,r}}v'_i \in V'_{\lambda_i}, 
  \quad
  \lambda_i = (\eta_i, \tau(\lambda'_{i,1}), \ldots, \tau(\lambda'_{i,r})) \in (P^\gp\otimes\CC)\times {\rm im}(\tau)^r,
\]
these elements generate $V$. Finally, we have $V' = V\otimes_{\CC[Q]}\CC[Q']$, so $V\to V\otimes_{\CC[Q]}\CC[Q']$ is injective. The construction $V'\mapsto V$ is functorial in $V'$ and produces a quasi-inverse to the given functor, showing (a). 

To show (b), suppose that $V'$ is freely generated by the homogeneous elements $v'_i\in V'_{\lambda'_i}$. Then $V$ is freely generated by the elements $v_i$. 
\end{proof}

\begin{proof}[{Proof of Theorem~\ref{thm:canonical-ext}}]
(a) The subcategory $\MIC^\tau(\ov{X},X/\CC)$ of $\MIC(\ov{X}/\CC)$ is defined by a pointwise condition and hence is of local nature on $\ov{X}$. Localizing near a point $x\in \ov{X}$ and using the local form of a good embedding, we may assume that $\ov{X}$ is an open subset of the form $U_\delta$ of $\mathbf{A}_{P\times\NN^r, K}$ appearing in Lemma~\ref{lem:local-RH}, so that $X = U'_\delta$. Using that lemma combined with Theorem~\ref{thm:ogus-RH} and \cite[Remark~2.1.3]{OgusRH}, we obtain a commutative square of categories and functors
\[ 
  \xymatrix{
    \MIC^\tau(\ov{X},X/\CC) \ar[r]^-{j^*} \ar[d]^{\displaystyle\wr} & \MIC(X/\CC) \ar[d]^{\displaystyle\wr} \\
    L^\tau(P\times \NN^r, K) \ar[r] & L(P\times \ZZ^r, K)
  }
\]
where $L^\tau(P\times \NN^r, K)$ is as in Lemma~\ref{lem:canonical-equivariant}. The bottom arrow is an equivalence by (a) of that lemma, and the assertion follows.

(b) Using the argument in (a), this follows from Lemma~\ref{lem:canonical-equivariant}(b).
\end{proof}

\section{Regular connections and the Riemann--Hilbert correspondence}
\label{s:logRH}

\subsection{Regular connections}
\label{ss:regular}

Let $X$ be an idealized smooth log scheme over $\CC$, let $\pi\colon X^\#\to X$ be the morphism constructed in Proposition~\ref{prop:splitting-functors}, and let $\eps_{\rm univ}$ be the universal splitting of the log structure $\cM_{X^\#} = \pi^*\cM_X$. The log scheme $X^\#$ is idealized smooth and hollow, and so its underlying scheme is smooth (Lemma~\ref{lem:hollow-smooth-basics}(a)). As in \S\ref{ss:log-conn-loc-const}, we have a functor
\[ 
  \epsast_{\rm univ} \colon \MIC(X^\#/\CC) \ra \MIC(\underline{X}^\#/\CC).
\]

\begin{defin} \label{def:regular}
  With notation as above, we say that an object $E$ of $\MIC(X/\CC)$ is \emph{regular} if the object $\epsast_{\rm univ}(\pi^* E)$ of $\MIC(\underline{X}^\#/\CC)$ is regular in the classical sense \cite[Chap.\ II, D\'efinition~4.2]{DeligneLNM163}. We denote by $\MIC_{\rm reg}(X/\CC)$ the full subcategory of $\MIC(X/\CC)$ consisting of the regular objects.
\end{defin}

We note right away that if $X$ has trivial log structure, then $X=X^\#=\underline{X}$ and an object $E$ of $\MIC(X/\CC) = \MIC(\underline{X}/\CC)$ is regular in the above sense if and only if it is regular in the classical sense.

The use of the universal splitting $\epsast_{\rm univ}$ on $X^\#$ is nicely canonical. However, any splitting will do:

\begin{lemma} \label{lem:regular-indep}
  Let $X$ be a hollow idealized smooth log scheme over $\CC$ and let $\eps_i$ ($i=0,1$) be two splittings on $X$. Let $E$ be an object of $\MIC(X/\CC)$. Then $\epsast_0(E)$ is regular if and only if $\epsast_1(E)$ is regular.
\end{lemma}

\begin{proof}
We may assume that $\oM_Y$ is constant with value $P$ and that $\underline{Y}$ admits a good compactification $\underline{Y}\hookrightarrow \ov Y$ (which in this case means the usual snc compactification). Write $\nabla_i = \epsast_i(\nabla)$ for the connection on $\epsast_i(E)$, and suppose that $\epsast_0(E)=(E,\nabla_0)$ is regular. Let $\ov{E}$ be the canonical extension of $(E, \nabla_0)$ to a connection on $\ov Y$ with logarithmic poles along $D = \ov Y\setminus \underline{Y}$. We claim that $\nabla_1$ extends to a logarithmic connection on $\ov{E}$ as well, showing that $\epsast_1(E)=(E, \nabla_1)$ is regular. 

By Lemma~\ref{lem:dependence-on-eps}, we have $\nabla_1 = \nabla_0 - (1\otimes\delta(\eps_0, \eps_1))\circ \rho_\nabla$. Pick a basis $p_1, \ldots, p_r$ of $P^\gp$, and write $f_i = \eps_0(p_i)/\eps_1(p_i)\in \Gamma(Y, \cO_Y^\times)$. Then $d\log(f_i)$ is a section of $\Omega^1_{\ov Y/\CC}(\log D)$. We explicate $\nabla_1$ as 
\[ 
  \nabla_1 = \nabla_0 - \sum_i \rho_i \otimes d\log(f_i)
\]
where $\rho_i\colon E\to E$ is the residue map induced by $p_i$. By functoriality of the canonical extension, since the maps $\rho_i\colon E\to E$ are horizontal with respect to $\nabla_0$, they extend uniquely to morphisms $\ov\rho_i\colon \ov{E}\to\ov{E}$. Therefore $\nabla_1$ maps $\ov E$ into $\ov E\otimes \Omega^1_{\ov Y/\CC}(\log D)$.
\end{proof}

\begin{lemma} \label{lem:regular-basic}
  Let $X$ be an idealized smooth log scheme over $\CC$ and let $E$ be an object of $\MIC(X/\CC)$. The following conditions are equivalent.
  \begin{enumerate}[(a)]
    \item $E$ is regular,
    \item for every idealized smooth and hollow $Y$ with a strict map $f\colon Y\to X$, and every splitting $\eps$ of $f^*\cM_X$, the object $\epsast(f^* E)$ of $\MIC(\underline{Y}/\CC)$ is (classically) regular,
    \item there exists a log dominant (Definition~\ref{def:log-dominant}) morphism $f\colon Y\to X$ with $Y$ idealized smooth and hollow and a splitting $\eps$ on $Y$ such that $\epsast(f^* E)$ is (classically) regular.
  \end{enumerate}
\end{lemma}

\begin{proof}
Every $f\colon Y\to X$ as in (b) uniquely factors through a map $\tilde f\colon Y\to X^\#$ such that the splitting $\eps$ is the preimage of the universal splitting $\eps_{\rm univ}$. Then
\[ 
  \epsast(f^* E) = \epsast(\tilde f^* \pi^* E) = \tilde f^* \epsast_{\rm univ}(\pi^* E)
\]
which shows that (a) implies (b). Moreover, (b) implies (c). 

It remains to show that (c) implies (a). We have a commutative square (see Remark~\ref{rmk:Xsplit-functorial})
\[
  \xymatrix{
    Y^\# \ar[d]_{\pi_Y} \ar[r]^g & X^\# \ar[d]^{\pi} \\
    Y\ar[r]_f & X
  }
\]
where the top map $g$ is a dominant as a map of schemes. Denote by $\eps_0$ and $\eps_1$ the splittings on $Y^\#$ induced by $\eps$ and $\eps_{\rm univ}$, respectively. By Lemma~\ref{lem:regular-indep}, we have 
\[ 
  \text{$\epsast_0(g^* \pi^* E)$ is regular}
  \quad\Leftrightarrow\quad
  \text{$\epsast_1(g^* \pi^* E)$ is regular} .
\]  
But $\epsast_0(g^* \pi^* E) \simeq \epsast_0(\pi_Y^* f^* E) = \pi_Y^* (\epsast(f^* E))$, which is regular since $\epsast(f^* E)$ is. Therefore 
\[ 
  \epsast_1(g^* \pi^* E) \simeq g^* (\epsast_{\rm univ}(\pi^* E))
\]
is regular, and hence so is $\epsast_{\rm univ}(\pi^* E)$ by \cite[Chap.\ II, Proposition~4.6(iii)]{DeligneLNM163}.
\end{proof}

\begin{cor} \label{cor:regular-basic}
  Let $X$ be an idealized smooth log scheme over $\CC$. Then the following hold.
  \begin{enumerate}[(a)]
    \item For every morphism $f\colon Y\to X$ such that $Y$ is idealized smooth, the pull-back functor 
    \[
      f^*\colon \MIC(X/\CC)\to\MIC(Y/\CC)
    \]
    maps $\MIC_{\rm reg}(X/\CC)$ into $\MIC_{\rm reg}(Y/\CC)$.
    \item If $f\colon Y\to X$ is a log dominant (Definition~\ref{def:log-dominant}) morphism of idealized smooth log schemes over $\CC$ and $E$ is an object of $\MIC(X/\CC)$ such that $f^* E$ is regular, then $E$ is regular as well.
  \end{enumerate}
\end{cor}

\begin{proof}
The first assertion follows from characterization (b) in Lemma~\ref{lem:regular-basic}, and the second one from characterization (c).
\end{proof}

\begin{prop} \label{prop:regularity-properties}
  Let $X$ be an idealized smooth log scheme over $\CC$. Then, the category $\MIC_{\rm reg}(X/\CC)$ is an abelian subcategory of $\MIC(X/\CC)$ stable under direct sums, tensor products, exterior and symmetric powers, duals, internal Hom, subobjects, quotients, and extensions.
\end{prop}

\begin{proof}
We will use the corresponding assertions about $\MIC_{\rm reg}(\underline{X}^\#/\CC)$, see \cite[Chap.\ II, Proposition~4.6]{DeligneLNM163}.

\medskip
\noindent {\sc Tensor operations.} Let $E$ and $F$ be regular. Since the functor $\epsast_{\rm univ} \pi^*$ is monoidal, i.e.\ 
\[ 
  \epsast_{\rm univ} (\pi^*(E\otimes F)) \simeq \epsast_{\rm univ} (\pi^* E) \otimes \epsast_{\rm univ} (\pi^* F),
\] 
the assertion follows from the corresponding assertion about $\MIC_{\rm reg}(\underline{X}^\#/\CC)$. Since exterior and symmetric powers are natural direct summands of tensor powers, we obtain the regularity of $\bigwedge^p (E)$ and ${\rm Sym}^p(E)$ as well.

\medskip
\noindent {\sc Quotients.} Since the functor $\epsast_{\rm univ} \pi^*$ is right exact (being the composition of the right exact $\pi^*$ and the exact functor $\epsast_{\rm univ}$), the assertion follows from the corresponding assertion about $\MIC_{\rm reg}(\underline{X}^\#/\CC)$. 

\medskip
\noindent {\sc Extensions.} Let $0\to E'\to E\to E''\to 0$ be a short exact sequence in $\MIC(X/\CC)$, with $E'$ and $E''$ regular. Applying the functor $\epsast_{\rm univ} \pi^*$, we obtain a right exact sequence in $\MIC_{\rm reg}(\underline{X}^\#/\CC)$
\[ 
  \xymatrix{ \epsast_{\rm univ} (\pi^*E') \ar[r] & \epsast_{\rm univ} (\pi^*E) \ar[r] & \epsast_{\rm univ} (\pi^*E'')\ar[r] & 0.}
\]
Therefore $\epsast_{\rm univ} (\pi^*E)$ is an extension of the regular $\epsast_{\rm univ} (\pi^*E'')$ by a quotient of the regular $\epsast_{\rm univ} (\pi^*E')$, and hence is regular.

\medskip
\noindent {\sc Subobjects.} Let $E'\to E$ be a monomorphism in $\MIC(X/\CC)$ where $E$ is regular. We want to show that $E'$ is regular as well. The subtlety is that $\pi\colon X^\#\to X$ is not flat, and so $\pi^*(E')\to \pi^*(E)$ might not be injective. 

Suppose first that $X$ is hollow. In this case, $X^\#\to X^\flat=X$ is smooth, and the assertion follows from the classical case since $\epsast_{\rm univ}(\pi^* E') \to \epsast_{\rm univ}(\pi^* E')$ is injective. 

The question is strict \'etale local, so we may assume $X$ is affine admits a strict \'etale map $X\to \mathbf{A}_{P,K}$. In fact, by Lemma~\ref{lem:local-embedding} we may reduce to the case $K=\emptyset$. For a face $F$ of $P$, let $U_F$ be the preimage of $\mathbf{A}_{P_F}$ in $X$ and let $X_F$ be the preimage of $\mathbf{A}_{F^\gp}$ embedded as in \eqref{eqn:P-Fgp-map}, so that
\[ 
  X^\flat = \mathbf{A}_P^\flat\times_{\mathbf{A}_P} X \simeq \coprod_F X_F.
\]
We need to show that the pull-back of $E'$ to every $X_F$ is regular. Fix a face $F\subseteq P$ and let $Z = X_F$ be the corresponding stratum. Since $X_F\subseteq U_F$, we may replace $X$ with $U_F$ and hence assume $Z$ is closed in $X$, cut out by an ideal $\mathcal{I}\subseteq \cO_X$. For $n\geq 0$, we let $Z^{(n)}$ be the closed subscheme cut out by the ideal $\mathcal{I}^{n+1}$. 

As in the proof of Proposition~\ref{prop:analytic-sections}, by the Artin--Rees lemma there exists a $k\geq 0$ and a surjection
\[ 
  {\rm im}(E'\to E|_{Z^{(k)}}) \ra E'|_Z.
\]
We may then replace $X$ with $Z^{(k)}$ and $E'$ with ${\rm im}(E'\to E|_{Z^{(k)}})$. We have thus reduced to the case $X$ with constant log structure.

We may reduce to the case $E' = \mathcal{I}E$ using the Snake Lemma applied to the diagram
\[ 
  \xymatrix{
    0\ar[r] & \mathcal{I} E' \ar[d]\ar[r] & E' \ar@{^{(}->}[d] \ar[r] & E'/\mathcal{I}E' \ar[d] \ar[r] & 0 \\
    0\ar[r] & \mathcal{I}E \ar[r] & E \ar[r] & E/\mathcal{I}E \ar[r] & 0. 
  }
\] 

However, $\mathcal{I}E$ itself is a quotient of $\mathcal{I}\otimes E$. Since tensor products of regular objects are regular, we are finally reduced to showing that $\mathcal{I}$ is regular as an object of $\MIC(X/\CC)$. 

To show that $\mathcal{I}$ is regular, we may assume that $X$ admits a strict \'etale map $X\to \mathbf{A}_{P,K}$ with $\mathcal{I}$ generated by $\sqrt{K} = P \setminus P^\times$. Thus $\mathcal{I}$ is the pullback of $\sqrt{K}\cdot \cO_{\mathbf{A}_{P,K}}$, and it suffices to treat $X=\mathbf{A}_{P,K}$ and $\mathcal{I}=\sqrt{K}\cdot \cO_{\mathbf{A}_{P,K}}$. Since $\ov{P}^\gp$ is free, we may write $P=\ov{P}\times P^\times$. Then $\mathcal{I}$ is the pullback of the ideal generated by $\ov{P}\setminus\{0\}$ in $\mathbf{A}_{\ov{P},\ov{K}}$ where $\ov{K}$ is the image of $K$ in $\ov{P}$. This reduces us further to the case $X=\Spec(\ov{P}\to \CC[\ov{P}]/(\ov K))$, but then $\underline{X}$ is a point, and we conclude by Lemma~\ref{lem:regular-basic}(c).

\medskip
\noindent {\sc Internal Hom.} As in the proof above, we reduce to showing that if $E$ and $F$ are regular objects of $\MIC(X/\CC)$ and $Z\subseteq X$ is a closed log stratum, then $\underline{\Hom}(E, F)|_Z$ is regular. If $X=Z$ (i.e.\ $X$ is hollow), this is clear because $\epsast_{\rm univ}(\pi^*\underline{\Hom}(E, F)) \simeq \underline{\Hom}(\epsast_{\rm univ}(\pi^* E), \epsast_{\rm univ}(\pi^* F))$ by flatness of $\pi$. In general, by the Mittag--Leffler property (Lemma~\ref{lem:mittag-leffler} below with $e=0$) there exists a $k\geq 0$ such that the natural map
\[ 
  \underline{\Hom}(E, F)|_Z \ra \underline{\Hom}(E|_{Z^{(k)}}, F|_{Z^{(k)}})|_Z
\]
is injective. Therefore we may replace $X$ with $Z^{(k)}$. In this case, the short exact sequences 
\[
  \xymatrix{0\ar[r] & \mathcal{I}^{j+1} F\ar[r] & \mathcal{I}^j F\ar[r] & \mathcal{I}^j F/\mathcal{I}^{j+1}F\ar[r] & 0}
\]
induce exact sequences
\[ 
  \xymatrix{\Hom(E, \mathcal{I}^{j+1} F)\ar[r] & \Hom(E, \mathcal{I}^j F)\ar[r] & \Hom(E, \mathcal{I}^j F/\mathcal{I}^{j+1}F)}
\] 
and by induction on $j$ we reduce to the case $\mathcal{I}F=0$. In this case, $\underline{\Hom}(E, F) = \underline{\Hom}(E/\mathcal{I}E, F) = \underline{\Hom}(E|_Z, F|_Z)$, and we reduced to the case $X=Z$ already handled before. 
\end{proof}

The above proof used the following lemma from commutative algebra.

\begin{lemma} \label{lem:mittag-leffler}
  Let $A$ be a Noetherian ring, let $I\subseteq A$ be an ideal, and let $E$ and $F$ be finitely generated $A$-modules. For $n\geq 0$, set $A_n = A/I^{n+1}$, and for an $A$-module $M$ we set $M_n = M\otimes_A A_n$. Then, for every $e\geq 0$ there exists a $k_e\geq e$ such that for every $k\geq k_e$ the morphism
  \[ 
    \Hom(E, F)_e \ra \Hom(E_k, F_k)_e
  \]
  (induced by the restriction map $\pi\colon \Hom(E, F)\to \Hom(E_k, F_k)$) is injective.
\end{lemma}

\begin{proof}
Fix a presentation $A^m\to A^n \to E\to 0$, so that $\Hom(E, F)$ is the cohomology module of 
\[
  0\ra F^n\ra F^m
\]
with $F^n\to F^m$ induced by the transpose of $A^m\to A^n$ (see \stacks{09BB}), and analogously $\Hom(E_k, F_k)$ is the cohomology module of $0\to F_k^n\to F_k^m$. By the Mittag--Leffler property \stacks{0EGU}, there exists a $k_e\geq e$ such that for $k\geq k_e$ there exists a $\psi_k$ making the triangle below commute
\[ 
  \xymatrix{
    & \Hom(E_k, F_k)\ar[d]^{\psi_k} \\
    \Hom(E, F)_k \ar[r] \ar[ur]^\pi & \Hom(E, F)_e.  
  }
\]
Tensoring the above diagram with $A_e$ shows that $\Hom(E, F)_e\to \Hom(E_k, F_k)_e$ is a split injection.
\end{proof}

\begin{remark} \label{rmk:birational-question}
  Let $X$ be an idealized smooth log scheme over $\CC$, let $E$ be an object of $\MIC(X/\CC)$, and let $U\subseteq X$ be an open subset containing the associated points of $E$. Suppose that $E|_U$ is regular. Must $E$ be regular as well?
\end{remark}

The following provides a useful criterion for checking regularity using formal curve germs. For a space of the form $T=\Spec(P\xto{\eps} \CCt)$ where $P$ is a sharp fs monoid and $\eps(P\setminus \{0\}) = 0$, we define 
\[
  \Omega^1_{T/\CC} = \CCt d\log(t)\oplus (P^\gp\otimes_\ZZ \CCt).
\]
We define $\MIC(T/\CC)$ in the obvious way, and ``projection onto $\CCt d\log(t)$'' defines a functor \[
  \epsast\colon \MIC(T/\CC)\ra \MIC(\CCt/\CC)
\]
to the category of finite dimensional vector spaces over $\CCt$ endowed with an action of $t\frac{d}{dt}$ satisfying the Leibniz rule. As in \cite[Chap.\ II, D\'efinition 1.11]{DeligneLNM163}, we say that an object $E$ of $\MIC(\CCt/\CC)$ is \emph{regular} if it admits a $\CCs$-lattice stable under the action of $t\frac{d}{dt}$. For an idealized smooth log scheme $X$ over $\CC$, a map $\gamma\colon T\to X$ induces a functor $\gamma^*\colon \MIC(X/\CC)\to \MIC(T/\CC)$. 

For $T=\Spec(P\xto{\eps} \CCt)$ as above and $X$ an idealized smooth log scheme over $\CC$, let us say that a map $\gamma\colon T\to X$ is \emph{algebraic} if it factors as $T\to C\to X$ where $C\hookrightarrow \ov{C}$ is a good embedding of idealized smooth log schemes of dimension one over $\CC$ with $C = \ov{C} \setminus \{y\}$, where $t$ is a local parameter of $\cO_{\ov{C},y}$ (identifying the fraction field of $\widehat{\cO}_{\ov{C},y}$ with $\CCt$), and where $P\oplus \NN\to\cO_{\ov{C},y}$ sending $(p,n)$ to $0$ if $p\neq 0$ and $(0,n)$ to $t^n$ is a local chart at $\ov{C}$, and where $T\to C$ is induced by these data.

\begin{prop} \label{prop:curve-germ-criterion}
  Let $X$ be an idealized smooth log scheme over $\CC$. An object $E$ of $\MIC(X/\CC)$ is regular if and only if for every sharp fs monoid $P$ and every strict morphism
  \[ 
    f\colon \Spec (P\xrightarrow{\eps} \CCt ) \ra X,
  \]
  the object $\epsast (f^* E)$ of $\MIC(\CCt/\CC)$ is regular. Moreover, it suffices to check this condition on algebraic maps $f$.
\end{prop}

\begin{proof}
Let $E\in \MIC(X/\CC)$. By definition, $E$ is regular if and only if the connection $\epsast_{\rm univ} (\pi^* E)$ on $\underline{X}^\#$ is regular. By \cite[Chap.\ II, Th\'eor\`eme~4.1]{DeligneLNM163}, this holds if and only if for every morphism
\[ 
  g\colon \Spec \CC(\!(s)\!) \ra \underline{X}^\#,
\]
the pull-back $g^*(\epsast_{\rm univ} (\pi^* E))$ is regular, and it suffices to consider such germs which are ``algebraic'' in our sense. By the definition of $X^\#$, such a map corresponds to a morphism $f\colon \Spec \CC(\!(s)\!)\to \underline X$ and a splitting $\eps$ of $f^* \cM_X$. Moreover, we have
\[ 
  g^*(\epsast_{\rm univ} (\pi^* E)) \simeq \epsast(f^* E).
\]
Let $h\colon \Spec \CCt\to \Spec \CC(\!(s)\!)$ be a finite extension such that $h^* f^* \oM_X$ is constant, so that $h^* f^* \oM_X$ admits a global chart. Then $\epsast(f^* E)$ is regular if and only if $h^*(\epsast(f^* E))$ is.
\end{proof}

\begin{prop} \label{prop:curve-germ-has-center}
  Let $X$ be an idealized smooth log scheme over $\CC$ and let 
   \[ 
    f\colon \Spec (P\xrightarrow{\eps} \CCt ) \ra X,
  \]
  be a strict morphism, where $P$ is a sharp fs monoid. If the underlying map of schemes extends to $\Spec \CCs$, then for every $E\in \MIC(X/\CC)$ the object $\epsast (f^* E)$ of $\MIC(\CCt/\CC)$ is regular.
\end{prop}

\begin{proof}
Let $\Delta = \Spec \CCs$ and $\Delta^* = \Spec \CCt$, and let $j\colon \Delta^*\hookrightarrow \Delta$ be the inclusion. Let $\bar f \colon \Delta\to \underline{X}$ be an extension of $\underline{f}$. Passing to an \'etale neighborhood of $x = \bar f(0)$ we may assume that $X$ admits a global chart $Q\to \cM_X$. 

Let $\cM_{\Delta^*} = f^* \cM_X$, with the given chart $P\to \CCt$, and let $\cM_\Delta = \bar f^* \cM_X$, with the natural chart $Q\to \CCs$. Finally, let $\cM'_\Delta = j_* \cM_{\Delta^*}$ be the push-forward log structure. Since $\cM_{\Delta^*} = j^* \cM_{\Delta}$, we have the unit map
\[ 
  \cM_\Delta \ra \cM'_\Delta,
\]
and hence maps of log schemes
\[ 
  \bar f\colon (\Delta, \cM'_\Delta) \ra (\Delta, \cM_\Delta) \ra X,
\]
where the second map is strict and the first is an isomorphism on $\Delta^*$. Therefore the object $f^* E$ of $\MIC(\Delta^*, \cM_\Delta)$ extends to an object $\bar f^*(E)$ of $\MIC(\Delta, \cM'_\Delta)$.

The push-forward log structure $\cM'_\Delta = j_* \cM_{\Delta^*}$ admits a chart $\bar\eps\colon P\oplus \NN\to \CCs$ sending $(0, n) \mapsto t^n$ and everything else to $0$. This induces a splitting of the short exact sequence
\[ 
  \xymatrix{ 0\ar[r] & \Omega^1_{\Delta^*/\CC}(\log 0) \ar[r] & \Omega^1_{(\Delta, \cM'_\Delta)} \ar[r] & \cO_{\Delta}\otimes_\ZZ P^\gp\ar[r] & 0,}
\]
and just as before a functor
\[ 
  \ov{\eps}^\circledast \colon \MIC(\Delta, \cM'_\Delta) \ra \MIC(\Delta, j_* \cO_{\Delta^*}^*). 
\]
Moreover, we have
\[ 
  j^*(\ov{\eps}^\circledast(\bar f^* E)) \simeq \epsast (f^* E),
\]
and we conclude that $\epsast (f^* E)$ extends to a log connection on the formal disc. It is therefore regular.
\end{proof}

Propositions~\ref{prop:curve-germ-criterion} and \ref{prop:curve-germ-has-center} together imply:

\begin{cor} \label{cor:regular-on-proper}
  Let $X$ be an idealized smooth log scheme over $\CC$ whose underlying scheme $\underline{X}$ is proper over $\CC$. Then $\MIC_{\rm reg}(X/\CC) = \MIC(X/\CC)$.
\end{cor}

Finally, we turn to some basic cohomology comparison results in the case $X$ has locally constant log structure. 

\begin{lemma} \label{lem:basic-coh-comparison}
  Let $X$ be a hollow idealized smooth log scheme over $\CC$ and let $E$ be an object of $\MIC_{\rm reg}(X/\CC)$. Then, $H^*_{\rm dR}(X, E) \simeq H^*_{\rm dR}(X_{\rm an}, E_{\rm an})$.
\end{lemma}

\begin{proof}
We may assume that $\oM_X$ is constant and that $X$ admits a splitting $\eps$. By Proposition~\ref{prop:Higgs-description}, we have a map of spectral sequences
\[ 
  \xymatrix{
    H^i_{\rm dR}(\underline{X}, \epsast E\otimes \bigwedge^j \mathcal{Q}_X) \ar[d]_{\alpha_{ij}} \ar@{}[r]|-{\displaystyle\Rightarrow} & H^{i+j}_{\rm dR}(X, E) \ar[d]^\alpha \\
    H^i_{\rm dR}(\underline{X}_{\rm an}, \epsast E_{\rm an}\otimes \bigwedge^j \mathcal{Q}_X) \ar@{}[r]|-{\displaystyle\Rightarrow} & H^{i+j}_{\rm dR}(X_{\rm an}, E_{\rm an}).
  }
\]
Since $\bigwedge^j \mathcal{Q}_X$ is constant, each $\epsast E\otimes \bigwedge^j \mathcal{Q}_X$ is an object of $\MIC_{\rm reg}(\underline{X}/\CC)$. Therefore by the classical comparison theorem \cite[Chap.\ II, Th\'eor\`eme 6.2]{DeligneLNM163}, the maps $\alpha_{ij}$ are isomorphisms, and hence so is $\alpha$.
\end{proof}

\begin{cor} \label{cor:almost-basic-comparison}
  For $X$ with constant log structure and $E$ regular, we have $H^*_{\rm dR}(X, E) \simeq H^*_{\rm dR}(X_{\rm an}, E_{\rm an})$.
\end{cor}

\begin{proof}
Use the five lemma and the extensions $0\to \mathcal{I}E\to E\to E/\mathcal{I}E\to 0$ where $\mathcal{I}$ is the ideal of $X_{\rm red}$ in $X$ and apply Lemma~\ref{lem:basic-coh-comparison} inductively.
\end{proof}

\begin{remark} \label{rmk:D-mod-reg-holo}
The strategy of using $X^\#$ and the functor $\epsast_{\rm univ}(\pi^*(-))$ to study $\MIC(X/\CC)$ applies similarly to categories of coherent log $D$-modules on $X$ \cite{KoppensteinerTalpo,Koppensteiner} (see Remark~\ref{rmks:epsast}.\ref{rmk:D-mod}). For example, if $Z\subseteq X$ is a closed subset, then its log dimension $\op{logdim}(Z)$ introduced in \cite{KoppensteinerTalpo} equals the dimension of $\pi^{-1}(Z)\subseteq X^\#$. Moreover, using Lemma~\ref{lem:xsharp-x-unram} one can show that an object $M$ of $\cat{Mod}_{\rm coh}(\mathcal{D}_X)$ is holonomic if and only if the object $\epsast_{\rm univ}(\pi^*(M))$ of $\cat{Mod}_{\rm coh}(\mathcal{D}_{\underline{X}^\#})$ is holonomic in the classical sense. It seems natural to call a coherent $\mathcal{D}_X$-module $M$ \emph{regular holonomic} if $\epsast_{\rm univ}(\pi^*(M))$ is regular holonomic. That said, we did not investigate this notion beyond the case of $\cO_X$-coherent $\mathcal{D}_X$-modules i.e.\ objects of $\MIC(X/\CC)$.
\end{remark}

\subsection{The Existence Theorem}
\label{ss:existence}

\begin{thm}[Existence theorem] \label{thm:existence}
  Let $X$ be an idealized smooth log scheme over $\CC$. Then the composite functor
  \[ 
    \MIC_{\rm reg}(X/\CC) \hookrightarrow \MIC(X/\CC) \xrightarrow{E\mapsto E_{\rm an}} \MIC(X_{\rm an}/\CC)
  \]
  is an equivalence of categories.
\end{thm}

\begin{proof}
We first show that the functor is fully faithful. For objects $E$ and $F$ of $\MIC_{\rm reg}(X/\CC)$ consider the object $H = \underline{\Hom}(E, F)$. It is again regular by Proposition~\ref{prop:regularity-properties}, and we have $\Hom(E, F) = H^0_{\rm dR}(X, H)$ and $\Hom(E_{\rm an}, F_{\rm an}) = H^0_{\rm dR}(X_{\rm an}, H_{\rm an})$. So we need to show that for a regular object $H$ we have $H^0_{\rm dR}(X, H)\simeq H^0_{\rm dR}(X_{\rm an}, H_{\rm an})$. To this end, Proposition~\ref{prop:analytic-sections} applied to the log stratification of $X$ allows us to reduce to the case when the log structure on $X$ is locally constant. In this case we apply Corollary~\ref{cor:almost-basic-comparison} to conclude.

It remains to show essential surjectivity. By Corollary~\ref{cor:regular-basic}(b), the assertion is strict \'etale local on $X$. Therefore by Theorem~\ref{thm:good-comp-exists} we may assume that there exists a good compactification $X\hookrightarrow \ov{X}$. Consider the commutative square of categories and functors
\[ 
  \xymatrix{
    \MIC_{\rm reg}(\ov{X}/\CC) \ar[r] \ar[d] & \MIC(\ov{X}_{\rm an}/\CC) \ar[d] \\
    \MIC_{\rm reg}(X/\CC) \ar[r] & \MIC(X_{\rm an}/\CC)
  }
\]
The top functor is an equivalence by Lemma~\ref{cor:regular-on-proper} and GAGA. The right functor is essentially surjective by Proposition~\ref{thm:canonical-ext}. So the bottom functor is essentially surjective.
\end{proof}

As a byproduct of the proof, we obtain another familiar-looking characterization of regular connections.

\begin{cor}[Regularity $=$ log extendability] \label{cor:log-extendable}
  Let $X$ be an idealized smooth log scheme over $\CC$ and let $E$ be an object of $\MIC(X/\CC)$. The following are equivalent:
  \begin{enumerate}[(a)]
    \item $E$ is regular,
    \item for every strict \'etale $f\colon U\to X$ and every good compactification $j\colon U\hookrightarrow \ov{U}$, there exists an object $\ov{E}$ of $\MIC(\ov{U}/\CC)$ such that $j^*\ov{E} \simeq f^* E$,
    \item there exists a strict \'etale cover $\{f_i\colon U_i\to X\}_{i\in I}$, good compactifications $j_i\colon U_i \to \ov{U}_i$, and objects $\ov{E}_i$ of $\MIC(\ov{U}_i/\CC)$ such that $j^*_i\ov{E}_i \simeq f_i^* E$ for every $i\in I$.
  \end{enumerate}
\end{cor}

\begin{remark}
Is there a characterization of regular logarithmic connections in terms of growth conditions akin to \cite[Chap.\ II, Th\'eor\`eme~4.1(iii)]{DeligneLNM163}?
\end{remark}

\subsection{The Comparison Theorem}
\label{ss:comparison}

\begin{thm}[Comparison theorem] \label{thm:comparison}
  Let $X$ be an idealized smooth log scheme over $\CC$ and let $E$ be an object of $\MIC_{\rm reg}(X/\CC)$. Then $H^*_{\rm dR}(X, E)\simeq H^*_{\rm dR}(X_{\rm an}, E_{\rm an})$.
\end{thm}

The proof will rely on slight variants of two results of Ogus, adapted to our notion of a good compactification.

\begin{thm}[{Variant of \cite[Theorem 3.4.9]{OgusRH}}] \label{thm:349variant}
  Let $j\colon X\hookrightarrow \ov{X}$ be a good embedding of idealized smooth log complex analytic spaces. Let $E$ be an object of $\MIC(\ov{X}/\CC)$, and suppose that the following two conditions are satisfied
  \begin{enumerate}
    \item For every $x\in X$, the set $R_x$ of exponents at infinity (\S\ref{ss:can-ext}) of $E$ at $x$ satisfies $R_x\cap \mathcal{F}^{\rm gp}_x \subseteq \mathcal{F}_x$. 
    \item $E\to j_* j^* E$ is injective.
  \end{enumerate}
  Then $H^*_{\rm dR}(\ov{X}, E)\simeq H^*_{\rm dR}(X, j^* E)$.
\end{thm}

\begin{proof}
As in \cite{OgusRH}, we may reduce to an analogous statement about monoids. Let $P$ be a sharp fs monoid and let $P\to P'$ be a localization with respect to a face $F\subseteq P$. Let $V$ be an object of $L(P)$ (Example~\ref{ex:L-of-AP}) and let $V' = V\otimes_{\CC[P]} \CC[P']$. Suppose that (1) $V\to V'$ is injective (i.e.\ $V$ is $F$-torsion free) and that (2) the $F^\gp$-graded $\CC[F]$-submodule $\tilde V = \bigoplus_{\lambda\in F^\gp} V_\lambda$ is generated by elements lying in gradings in $-F\subseteq F^\gp$. The statement we need (itself a variant of \cite[Corollary~1.4.6]{OgusRH}) is that the map on group cohomology
\[
  H^*(\pi_1(P), V_0) \ra H^*(\pi_1(P), V'_0)
\]
is an isomorphism. As in the proof of \cite[Corollary~1.4.6]{OgusRH}, it suffices to observe that in fact $V_0=V'_0$ (variant of \cite[Corollary~1.1.3]{OgusRH}). By assumption (1), the map $V_0\to V'_0$ is injective, so we need to prove surjectivity. Let $f$ be a generator of $F$ as a face, so that for every $x\in F$ there exists an $y\in F$ and an integer $n\geq 0$ such that $x+y=nf$. Then $V' = V[1/f]$ and writing the localization as a colimit we have
\[ 
  V'_0 = \varinjlim_{n}\, V_{nf}.
\]
Let $v'\in V'_0$, and represent it by an element $v'\in V_{nf}$ for some $n\geq 0$. By assumption (2) there exist homogeneous generators  $v_1, \ldots, v_s$ of $\tilde{V}$ with $v_i\in V_{-f_i}$ for some $f_i\in F$. Thus, we can write $v' = \sum g_i v_i$ with $g_i\in \CC[F]_{nf+f_i}$. We may write $g_i = \alpha_i f^n f_i$ where $\alpha_i\in\CC$, and set $v = \sum \alpha_i f_i v_i \in V_0$. Then $v' = f^n v$, and so $v'$ and $v$ have the same image in $V'_0$.
\end{proof}

\begin{thm}[{Variant of \cite[Theorem 3.4.10]{OgusRH}}] \label{thm:3410variant}
  Let $j\colon X\hookrightarrow \ov{X}$ be a good embedding of idealized smooth log schemes over $\CC$ and let $E$ be an object of $\MIC(\ov{X}_{\rm an}/\CC)$. Denote by $j_* j_m^*$ the functor of meromorphic sections as in \cite[Chap.\ II, 3.12]{DeligneLNM163}. Then the inclusion
  \[ 
    j_* j_m^* (E\otimes\Omega^\bullet_{\ov{X}/\CC}) \ra j_* j^* (E\otimes\Omega^\bullet_{\ov{X}/\CC})
  \] 
  is a quasi-isomorphism.
\end{thm}

\begin{proof}
Repeat the original proof, replacing the use of \cite[Theorem 3.4.9]{OgusRH} with Theorem~\ref{thm:349variant}.
\end{proof}

\begin{proof}[{Proof of Theorem~\ref{thm:comparison}}]
By cohomological descent, the statement is strict \'etale local on $X$, so we may assume that $X$ admits a good compactification $j\colon X\to \ov{X}$. By Corollary~\ref{cor:log-extendable} of the Existence Theorem, $E$ extends to an object $\ov{E}$ of $\MIC_{\rm reg}(\ov{X}/\CC) = \MIC(X/\CC)$. Since $j$ is affine and $j_{\rm an}$ is Stein, arguing as in \cite[Chap.\ II, \S 6.6]{DeligneLNM163} we obtain
\[
  H^*_{\rm dR}(X, E) \simeq H^*(\ov{X}_{\rm an}, j_* j_m^* (E\otimes\Omega^\bullet_{\ov{X}/\CC}))
  \quad\text{and}\quad
  H^*_{\rm dR}(X_{\rm an}, E_{\rm an}) \simeq H^*(\ov{X}_{\rm an}, j_* j^* (E\otimes\Omega^\bullet_{\ov{X}/\CC})).
\]
It remains to invoke Theorem~\ref{thm:3410variant} to conclude.
\end{proof}

\subsection{The Regularity Theorem}
\label{ss:regularity}

We recall the Katz--Oda construction of the Gauss--Manin connection in the relevant context (see e.g.\ \cite{IllusieKatoNakayama}). Let $f\colon Y\to X$ be a proper smooth  morphism of idealized smooth log schemes over $\CC$. We have the short exact sequence
\[ 
  \xymatrix{0\ar[r] & f^* \Omega^1_{X/\CC} \ar[r] & \Omega^1_{Y/\CC} \ar[r] & \Omega^1_{Y/X} \ar[r] & 0}
\]
from which we obtain a decreasing filtration on $\Omega^q_{Y/\CC} = \bigwedge^q \Omega^1_{Y/\CC}$
\[ 
  K^r \Omega^q_{Y/\CC} = {\rm im}\left( f^*\Omega^{r}_{X/\CC}\otimes \Omega^{q-r}_{Y/\CC}\to \Omega^q_{Y/\CC}\right)
\]
satisfying
\[
  {\rm gr}^r \Omega^q_{Y/\CC} = K^r\Omega^q_{Y/\CC}/K^{r+1}\Omega^q_{Y/\CC} \simeq f^*\Omega^{r}_{X/\CC}\otimes \Omega^{q-r}_{Y/X} .
\]
Let $E$ be an object of $\MIC(Y/\CC)$. Then for every $r\geq 0$, $E\otimes K_r\Omega^\bullet_{Y/\CC}$ is a subcomplex of $E\otimes\Omega^\bullet_{Y/\CC}$ and 
\[
  \op{gr}^r(E\otimes\Omega^\bullet_{Y/\CC}) \simeq f^*\Omega^r_{X/\CC}\otimes (E\otimes \Omega^{\bullet-r}_{Y/X}).
\]
In particular we have a short exact sequence of complexes
\[ 
  \xymatrix{0\ar[r] & f^*\Omega^1_{X/\CC}\otimes (E\otimes \Omega^{\bullet-1}_{Y/X}) \ar[r] & {\displaystyle \frac{E\otimes \Omega^\bullet_{Y/\CC}}{E\otimes K_2\Omega^\bullet_{Y/\CC}}} \ar[r] & E\otimes \Omega^\bullet_{Y/X} \ar[r] & 0.}
\]
Applying $Rf_*$ and the projection formula, we obtain as the connecting homomorphism 
\[ 
  \nabla_{\rm GM}\colon R^n f_* (E\otimes \Omega^\bullet_{Y/X}) \ra R^n f_* (E\otimes \Omega^\bullet_{Y/X}) \otimes \Omega^1_{X/\CC}.
\]
As in the classical case, one shows that this connection is integrable, which makes $R^n f_* (E\otimes \Omega^\bullet_{Y/X})$ into an object of $\MIC(X/\CC)$.

The analogous construction applies to smooth proper morphisms of idealized smooth log complex analytic spaces, in particular to $f_{\rm an}$. This way, we obtain an object $R^n f_{\rm an,*} (E_{\rm an}\otimes \Omega^\bullet_{Y_{\rm an}/X_{\rm an}})$ of $\MIC(X_{\rm an}/\CC)$. We have a comparison morphism
\begin{equation} \label{eqn:alg-an-GM} 
  (R^n f_* (E\otimes \Omega^\bullet_{Y/X}))_{\rm an} 
  \ra
  R^n f_{\rm an,*} (E_{\rm an}\otimes \Omega^\bullet_{Y_{\rm an}/X_{\rm an}})
\end{equation}
which is an isomorphism because $f$ is proper.

\begin{remark} \label{rmk:gray}
The main result of the PhD thesis of Aaron Gray \cite{Gray} (unpublished) asserts that under certain assumptions, the formation of higher direct images is compatible with Ogus' Riemann--Hilbert correspondence $\mathcal{V}$ (Theorem~\ref{thm:ogus-RH}). More precisely, for a proper and smooth morphism of smooth log complex analytic spaces $f\colon Y\to X$ such that the map $f^*\oM_Y^\gp\to\oM_X^\gp$ is injective and its cokernel is torsion-free, there is a natural isomorphism between the two compositions in the square 
\begin{equation} \label{eqn:gray}
  \xymatrix@C=2.5cm{
    \MIC(Y/\CC) \ar[r]_-{\sim}^-{\mathcal{V}_Y} \ar[d]_{R^n f_*((-)\otimes\Omega^\bullet_{Y/X})} & L(Y) \ar[d]^{R^n f_*^\Lambda} \\
    \MIC(X/\CC) \ar[r]^-\sim_-{\mathcal{V}_X} & L(X).  
  }
\end{equation}
where the functor $R^n f^\Lambda_*$ is constructed as follows. Recall (Definition~\ref{def:ogus-Lcoh}) that an object $V$ of $L(Y)$ is a $\Lambda_Y$-graded $\CC_Y^{\rm log}$-module satisfying certain conditions. Via $\CC_X^{\rm log}\to f^{\rm log}_* \CC_Y^{\rm log}$, the sheaf $R^n f^{\rm log}_* V$ is a $f_* \Lambda_Y$-graded $\CC_X^{\rm log}$-module. Using the map $\iota\colon \Lambda_X\to f^{\rm log}_* \Lambda_Y$, we define the $\Lambda_X$-graded $\CC_X^{\rm log}$-module $R^n f^\Lambda_* V$ by $(R^n f^\Lambda_* V)_\lambda = (R^n f^{\rm log}_* V)_{\iota(\lambda)}$ for a local section $\lambda$ of $\Lambda_X$. Then $R^n f^\Lambda_* V$ is an object of $L(X)$. 

Recently, Gray's results have been extended to log $D$-modules by Koppensteiner \cite[Proposition~6.4]{Koppensteiner}.
\end{remark}

The Regularity Theorem, stated below, asserts the regularity of $R^n f_* (E\otimes \Omega^\bullet_{Y/X})$ assuming that $E$ is regular. It is conditional on the following logarithmic variant of semistable reduction. We hope it is within reach of the current methods \cite{MotzkinTemkin,AdiprasitoLiuTemkin,Wlodarczyk}.

\begin{conj}[Log semistable reduction over a hollow curve] \label{conj:log-ss-reduction}
  Let $X$ be a hollow idealized smooth log scheme over $\CC$ of dimension $\leq 1$ and let $Y\to X$ be a smooth, proper, and exact morphism. Then, possibly after passing to a Kummer \'etale cover of $X$, there exists a commutative square
  \[ 
    \xymatrix{
      Y \ar@{^{(}->}[r] \ar[d] & \ov Y \ar[d] \\
      X \ar@{^{(}->}[r] \ar[r] & \ov X
    }
  \]
  where $X\hookrightarrow \ov X$ and $Y\hookrightarrow \ov Y$ are good compactifications (Definition~\ref{def:good-embedding}) and where $\ov Y\to \ov X$ is smooth.
\end{conj}

\begin{thm}[Regularity theorem] \label{thm:regularity}
  Let $f\colon Y\to X$ be a proper, smooth and exact morphism of idealized smooth log schemes over $\CC$ and let $E$ be an object of $\MIC_{\rm reg}(Y/\CC)$. Assume that Conjecture~\ref{conj:log-ss-reduction} holds. Then $R^n f_*(E\otimes\Omega^\bullet_{Y/X})$ are objects of $\MIC_{\rm reg}(X/\CC)$ for all $n\geq 0$.
\end{thm}

\begin{proof}
\emph{Step 1 --- reduction to $X$ with constant log structure.}
Set $V^n = R^n f_*(E\otimes \Omega^\bullet_{Y/X})$. We need to show that for every log stratum $Z\subseteq X$ (a connected component of $X^\flat$), the restrictions $V^n|_Z$ are regular. Since $Z$ is locally closed in $X$, we may replace $X$ with an affine open subset in which $Z\cap X$ is closed. Let $\mathcal{I}$ be the ideal of the closed immersion $Z\hookrightarrow X$, and let $Z^{(k)} = V(\mathcal{I}^{k+1})$ be the $k$-th order thickening of $Z$ and let
\[ 
  V^{n,k} = R^n (f|_{Z^{(k)}})_*(E|_{f^{-1}(Z^{(k)})}\otimes \Omega^\bullet_{Y/X}).
\] 
As in \cite[\S 4.1]{EGA_III_1}, the inverse system of coherent $\cO_X$-modules $V^{n,k}$ satisfies the Mittag--Leffler property, and if we set $R_k = \ker(V^n\to V^{n,k})$, then the submodules $\{R_k\}$ define the $\mathcal{I}$-adic topology on $V^n$. In particular, for $k$ large enough we have $\mathcal{I}V^n\supseteq R_k$, and hence a factorization
\[ 
  \xymatrix{
    V^n \ar@{->>}[r] \ar@{->>}[dr] & V^n/R_k \ar@{->>}[d] \ar@{^{(}->}[r] & V^{n,k} \\
    & V^n/\mathcal{I}V^n \ar@{=}[r] & V^n|_Z.
  }
\]
Thus $V^n|_Z$ is a subquotient of $V^{n,k}$ and by Proposition~\ref{prop:regularity-properties} it suffices to show that $V^{n,k}$ is regular. We may therefore replace $X$ with $Z^{(k)}$. 

\emph{Step 2 --- reduction to $X$ hollow.}
Let $X$ have locally constant log structure and let $Z = X_{\rm red}$ be the largest hollow subscheme of $X$, cut out by an ideal $\mathcal{I}$. Suppose that the assertion holds for $X=Z$. We show by induction on $m\geq 0$ that if $E$ is annihilated by $\mathcal{I}^{m+1}$, then the higher direct images $R^n f_*(E\otimes \Omega^\bullet_{Y/X})$ are regular for all $n\geq 0$. The base case $n=0$ corresponds to $E$ supported on $Z$, which follows from the case $X=Z$. For the induction step, use the long exact sequence obtained by applying $R^n f_*((-)\otimes\Omega^\bullet_{Y/X})$ to the short exact sequence in $\MIC_{\rm reg}(Y/\CC)$
\[ 
  \xymatrix{0\ar[r] & \mathcal{I}E\ar[r] & E\ar[r] & E/\mathcal{I}E\ar[r] & 0.}
\]

\emph{Step 3 --- ensuring base change.}
Suppose that $X$ is hollow. Since $X$ is reduced, there is a dense open $U\subseteq X$ on which the sheaves $R^j f_* (E\otimes \Omega_{Y/X}^i)$ are locally free and (hence) with formation commuting with base change for all $i,j\geq 0$. From the spectral sequence 
\[
  E_1^{ij}=R^j f_* (E\otimes \Omega_{Y/X}^i) 
  \quad\Rightarrow\quad
  R^{i+j} f_*(E\otimes\Omega^\bullet_{Y/X}),
\]
we deduce that the formation of $R^n f_*(E\otimes\Omega^\bullet_{Y/X})$ commutes with base change along strict morphisms $X'\to X$ with $X'$ idealized smooth over $\CC$ such that ${\rm im}(X'\to X)\subseteq U$. Recall that by Corollary~\ref{cor:regular-basic}(b) to prove the assertion we may replace $X$ with a dense open subset. We may therefore assume that the base change property holds with $U=X$.

\emph{Step 4 --- reduction to the case $X$ being a curve.}
Suppose that $X$ is hollow and that the formation of the $R^n f_*(E\otimes\Omega^\bullet_{Y/X})$ commutes with base change as above. By the criterion of Proposition~\ref{prop:curve-germ-criterion} it suffices to show that if $\gamma\colon C\to X$ is a strict map where $C$ smooth curve, then $\gamma^*(R^n f_*(E\otimes \Omega^\bullet_{Y/X}))$ is regular for all $n\geq 0$. By the base change property, we have
\[ 
  \gamma^*(R^n f_*(E\otimes \Omega^\bullet_{Y/X})) \simeq R^n (f_C)_* (E_{Y_C}\otimes \Omega^\bullet_{Y_C/C}).
\]
Therefore, if the assertion holds for all such $C\to X$, then it holds for $X$.

\emph{Step 5 --- semistable reduction.}
Suppose that $X$ is a hollow curve. By Conjecture~\ref{conj:log-ss-reduction}, after passing to a Kummer \'etale cover the map $Y\to X$ extends to a smooth morphism of good compactifications $\ov f\colon \ov Y\to \ov X$. By Corollary~\ref{cor:log-extendable}, $E$ extends to an object $\ov{E}$ of $\MIC_{\rm reg}(\ov Y/\CC)$. Since $Y=\ov Y\times_{\ov X} X$, we have $R^n f_*(E\otimes \Omega^\bullet_{Y/X}) \simeq (R^n \ov{f}_*(\ov E\otimes \Omega^\bullet_{\ov Y/\ov X})|_X$. Therefore $R^n f_*(E\otimes \Omega^\bullet_{Y/X})$ extends to the compactification $\ov X$ and is therefore regular (Corollary~\ref{cor:log-extendable} again).
\end{proof}

\begin{remark} \label{rmk:base-change}
In the situation of the classical Regularity Theorem, it is also true that the formation of $R^n f_* (E\otimes\Omega^\bullet_{Y/X})$ commutes with base change \cite[Chap.\ II, Proposition~6.14]{DeligneLNM163}. This is not true in general for proper smooth morphisms of smooth log schemes over $\CC$, as the following example shows.

Let $X=\mathbf{A}_{\NN^2}$ and let $f\colon Y\to X$ be the blowup of the origin $P$. Let $D\subseteq Y$ be the exceptional divisor, and set $E = \cO_Y(D)$. It inherits a logarithmic connection, as its dual $\mathcal{I}_D$ is a subobject of $\cO_Y$ in $\MIC(Y/\CC)$. As $f$ is \'etale, we have $\Omega^\bullet_{Y/X}=\cO_Y$ and $R^0 f_*(E\otimes\Omega^1_{Y/X}) = f_* \cO_Y(D) = \cO_X$. However, we have $H^0(f^{-1}(P), E) = H^0(\mathbf{P}^1, \cO_{\mathbf{P}^1}(-1)) = 0$, which violates base change along the strict inclusion $P\hookrightarrow X$.

However, assuming that the map $f^*\oM_X^\gp\to \oM_Y^\gp$ is injective and its cokernel is torsion-free, we can use Gray's results \cite{Gray} (Remark~\ref{rmk:gray}) and proper base change for maps of topological spaces to show the base change property for $R^n f_* (E\otimes\Omega^\bullet_{Y/X})$ along a strict map $g\colon X'\to X$ as follows. Since \eqref{eqn:alg-an-GM} is an isomorphism, it is enough to show the base change property in the analytic context, so changing notation from now on $Y\to X$ is a smooth proper morphism of idealized smooth complex analytic spaces satisfying the assumption on $f^*\oM_X^\gp\to \oM_Y^\gp$ and $g\colon X'\to X$ is a strict map of idealized smooth log complex analytic spaces. We set $Y' = Y\times_X X'$. As $g$ is strict, the four squares marked $\square$ in the diagram below are cartesian squares of topological spaces.
\[ 
  \xymatrix@R=.5cm{
    Y'_{\rm log} \ar[rrr]^{f'_{\rm log}} \ar[ddd]_{\tau_{Y'}} \ar[dr]|{g'_{\rm log}} & \ar@{}[rd]|{\displaystyle\square} & & X'_{\rm log} \ar[dl]|{g_{\rm log}} \ar[ddd]^{\tau_{X'}} \\
    \ar@{}[rd]|{\displaystyle\square} & Y_{\rm log} \ar[r]_{f_{\rm log}} \ar[d]^{\tau_Y} & X_{\rm log} \ar[d]_{\tau_X} \ar@{}[rd]|{\displaystyle\square} & \\
    & Y \ar[r]^f \ar@{}[rd]|{\displaystyle\square} & X & \\
    Y' \ar[ur]|{g'} \ar[rrr]_{f'} & & & X' \ar[ul]|g
  }
\]
Moreover, we have $\CC_{Y'}^{\rm log} = g_{\rm log}^* \CC_Y^{\rm log}$ and $\CC_{X'}^{\rm log} = g_{\rm log}^* \CC_X^{\rm log}$ (as $X$ and $X'$ are hollow). Therefore the functors $g^*\colon L(Y)\to L(Y')$ and $g^*\colon L(X)\to L(X')$ are just sheaf-theoretic pullback. By proper base change applied to $\tau_X f_{\rm log}$, for an object $V$ of $L(Y)$ we have
\begin{equation} \label{eqn:pbc}
  g^*_{\rm log}(Rf^\Lambda_* V) \isomto Rf'^\Lambda_* (g'^*_{\rm log} V).  
\end{equation}
Therefore for an object $E$ of $\MIC(Y/\CC)$, the base change map factors into a chain of isomorphisms
\[
  \xymatrix@C=2cm{
    g^* R^n f_*(E\otimes\Omega^\bullet_{Y/X}) \ar[d] \ar[r]_{\displaystyle\sim}^{\eqref{eqn:gray}} &  g^* \mathcal{V}_X^{-1} R^n f^\Lambda_* \mathcal{V}_Y(E) \ar[r]_-{\displaystyle\sim}^-{\eqref{eqn:ogus-RH-functoriality}} & \mathcal{V}_{X'}^{-1} g_{\rm log}^* R^n f^\Lambda_* \mathcal{V}_Y(E) \ar[d]_{\displaystyle\wr}^{\eqref{eqn:pbc}} \\
    R^n f'_*(g'^* E\otimes \Omega^\bullet_{Y'/X'}) & \mathcal{V}_{X'}^{-1} Rf'^\Lambda_* \mathcal{V}_{Y'} g'^* E \ar[l]^-{\displaystyle\sim}_-{\eqref{eqn:gray}} & \mathcal{V}_{X'}^{-1} Rf'^\Lambda_* g'^*_{\rm log} \ar[l]^-{\displaystyle\sim}_-{\eqref{eqn:ogus-RH-functoriality}}
  }
\]
\end{remark}

\section{Connections and local systems}
\label{ss:locsys}

\begin{defin} \label{def:basic}
  An idealized smooth log scheme over $\CC$ or log complex analytic space $X$ is \emph{of nc (normal crossings) type} if the stalks of $\oM_X$ are free monoids. 
\end{defin}

Let $X$ be log complex analytic space of nc type. For an object $V$ of $L(X)$ and $x\in X$, we denote by $R_x\subseteq \oM_{X,x}\otimes\CC$ the set of exponents of $V$ at $x$ \cite[Definition~2.1.1]{OgusRH}. Picking an isomorphism $\oM_{X,x}\simeq \NN^{r_x}$, we may view $R_x$ as a subset of $\CC^{r_x}$, well defined up to coordinate permutation.

\begin{defin} \label{def:basic-tau-adapted}
  Let $X$ be a log complex analytic space of nc type and let $\tau\colon \CC/\ZZ\to \CC$ be a section of the projection $\CC\to\CC/\ZZ$. We say that an object $V$ of $L(X)$ is \emph{$\tau$-adapted} if 
  \begin{enumerate}
    \item for every $x\in X$, the set $R_x$ of exponents of $V$ at $x$ is contained in ${\rm im}(\tau)^{r_x}$, and
    \item $V$ is locally free as a module over $\CC_X^{\rm log}$.
  \end{enumerate}
  We denote by $L^\tau(X)$ the full subcategory of $L(X)$ consisting of $\tau$-adapted objects.
\end{defin}

On the space $X_{\rm log}$, we have a surjective map of sheaves of rings
\[ 
  \CC_X^{\rm log} = \CC[-\tau^*\oM_X]/(-K_X) \ra \CC
\]
sending the nonzero elements of $\oM_X$ to zero. This defines a functor 
\[
  V\mapsto\underline{V} = V\otimes_{\CC_X^{\rm log}} \CC
\]
from the category $L(X)$ to the category of log-constructible sheaves of $\CC$-vector spaces on $X_{\rm log}$, sending locally free objects to locally constant sheaves. 
 
\begin{thm} \label{thm:basic-local-systems}
  Let $X$ be a log complex analytic space of nc type and let $\tau\colon \CC/\ZZ\to \CC$ be a section of the projection $\CC\to\CC/\ZZ$. 
  \begin{enumerate}[(a)]
    \item The functor $V\mapsto \underline{V}$ induces an equivalence of categories
    \[
      L^\tau(X)\isomto \cat{LocSys}_\CC(X_{\rm log}).
    \]
    \item If $\tau(0)=0$, then for every object $V$ of $L^\tau(X)$ the map  $V_0\hookrightarrow V\to \underline{V}$ induces an isomorphism
    \[
      H^*(X_{\rm log}, V_0)\isomto H^*(X_{\rm log}, \underline{V}).
    \]
  \end{enumerate}
\end{thm}

\begin{proof}
We describe the inverse functor. Let $W$ be a complex local system on $X^{\rm log}$. We set $V = W\otimes_\CC \CC_X^{\rm log}$. This is a locally free $\CC_X^{\rm log}$-module satisfying conditions (1), (2), and (3) of Definition~\ref{def:ogus-Lcoh}. In order to make it into an object of $L^\tau(X)$, we need to endow it with a $\bLambda_X$-grading satisfying condition (4) of Definition~\ref{def:ogus-Lcoh} and (2) of Definition~\ref{def:basic-tau-adapted}. Let $x\in X^{\rm log}$, and pick an isomorphism $\oM_{X,\tau(x)} \simeq \NN^{r_x}$. The stalk $W_x$ is a representation of 
\[
  \pi_1(\tau^{-1}(\tau(x)), x) = \Hom(\oM^\gp_{X,\tau(x)}, \ZZ(1)) = \ZZ(1)^{r_x}. 
\]
Let $W_x = \bigoplus W_{x,\chi}$ be the decomposition of $W_x$ into generalized eigenspaces, the direct sum taken over 
\[
  \chi\in \Hom(\pi_1(\tau^{-1}(\tau(x)), x), \CC^\times) = \Hom(\ZZ(1)^{r_x}, \CC^\times) = (\CC/\ZZ)^{r_x}. 
\]
Thus $V = \bigoplus_\chi V_\chi$ where $V_\chi = W_\chi\otimes_\CC \CC_X^{\rm log}$. We have $\bLambda_{X,x} = \CC^{r_x}$, and we endow $V_\chi$ with the $\CC^{r_x}$-grading where $w\otimes 1$ has degree $\tau(\chi)$ where $\tau\colon (\CC/\ZZ)^{r_x}\to \CC^{r_x}$ is the map induced by $\tau$. It is clear that this grading is independent of the choice of generators of $\oM_{X,\tau{x}}$, that it satisfies the two required conditions, and that the $\bLambda_{X,x}$-gradings on $V_x$ for all $x\in X^{\rm log}$ give rise to a $\bLambda_{X}$-grading on $V$. Trivially, $\underline{V}=W$. Moreover, the object associated to $W=\underline{V}$ is $V$, as $V$ is freely generated by $\bigoplus_{\lambda\in{\rm im}(\tau)^{r_x}} V_\lambda=W$ (cf.\ the proof of Lemma~\ref{lem:canonical-equivariant}).

For the second assertion, it is enough to check that the map $R\tau_{X,*}V_0\to R\tau_{X,*}\underline{V}$ is a quasi-isomorphism. By proper base change, it is enough to show that for every $x\in X$ we have 
\[ 
  H^*(\tau^{-1}(x), V_0) \isomto H^*(\tau^{-1}(x), \underline{V}).
\]
Pick an isomorphism $\cM_{X,x}\simeq \NN^{r_x}$ and write $\underline{V} = \bigoplus_\chi \underline{V}_\chi$ as in the previous paragraph. Then, for a base point $x'\in \tau^{-1}(x)$ 
\[
  H^*(\tau^{-1}(x), \underline{V}_\chi) = H^*(\ZZ(1)^{r_x}, V_{\tau(\chi),x'}) = 0
  \quad \text{for $\chi \neq 0$}
\]
(see the last assertion of \cite[Proposition~1.4.3]{OgusRH}). Therefore
\[ 
  H^*(\tau^{-1}(x), \underline{V}) = H^*(\tau^{-1}(x), \underline{V}_0) = H^*(\tau^{-1}(x), V_{\tau(0)}) = H^*(\tau^{-1}(x), V_0). \qedhere 
\]
\end{proof}

\begin{cor} \label{cor:basic-connections}
  Let $X$ be a log scheme of nc type over $\CC$ and let $\tau\colon \CC/\ZZ\to \CC$ be a section of the projection $\CC\to\CC/\ZZ$. Denote by $\MIC^\tau_{\rm reg}(X/\CC)$ the full subcategory of $\MIC_{\rm reg}(X/\CC)$ consisting of objects which are locally free (as $\cO_X$-modules) and whose exponents \cite[Definition~2.1.1]{OgusRH} belong to the image of $\tau$. 
  \begin{enumerate}[(a)]
    \item Then there is an equivalence of categories
    \[ 
      E\mapsto \underline{\mathcal{V}}(E)\quad \colon\quad \MIC^\tau_{\rm reg}(X/\CC) \isomto \cat{LocSys}_\CC(X_{\rm log}).
    \]
    \item If $\tau(0)=0$, then for every object $E$ of $\MIC^\tau_{\rm reg}(X/\CC)$ we have 
    \[
      H^*_{\rm dR}(X, E)\simeq H^*(X_{\rm log}, \underline{\mathcal{V}}(E)).
    \]
  \end{enumerate}
\end{cor}

We warn the reader that neither of the functors described in Theorem~\ref{thm:basic-local-systems} and Corollary~\ref{cor:basic-connections} is compatible with tensor product (actually, the subcategories $L^\tau(X)$ and $\MIC^\tau_{\rm reg}(X/\CC)$ are not closed under tensor product).

\bibliographystyle{plain} 
\bibliography{bib}

\end{document}